\documentclass[11pt,a4paper]{article}
\usepackage[a4paper,margin=30mm]{geometry}

\usepackage[T1]{fontenc}
\usepackage[utf8]{inputenc}
\usepackage{lmodern}
\usepackage{microtype}

\usepackage{amsmath,amssymb,amsfonts,amsthm}
\usepackage{bm,stmaryrd,euscript}
\usepackage{graphicx}
\usepackage{subcaption}
\usepackage{placeins}
\usepackage{cite}
\usepackage[hidelinks]{hyperref}
\usepackage{xcolor}
\usepackage{tikz}
\usetikzlibrary{arrows.meta,angles,quotes,patterns}

\newtheorem{theorem}{Theorem}[section]

\newtheorem{lemma}{Lemma}[section]

\theoremstyle{definition}
\newtheorem{assumption}{Assumption}[section]

\theoremstyle{remark}

\newtheorem{myremark}{Remark}[section]

\numberwithin{equation}{section}
\allowdisplaybreaks
\definecolor{green}{rgb}{0.0,0.50,0.0}
\tikzset{>={Straight Barb[angle'=80, scale=1.1]}}
\newcommand{\vertiii}[1]{{\left\vert\kern-0.25ex\left\vert\kern-0.25ex\left\vert #1 
    \right\vert\kern-0.25ex\right\vert\kern-0.25ex\right\vert}}
\newcommand{\tDelta}{\tilde{\Delta}_h}

\def\R {{\mathbb R}}

\def\le{\leqslant}
\def\ge{\geqslant}
\def\Omega{\varOmega}
\def\Delta{\varDelta}

\title{\bf 
An \({\bf H^{-1}}\) 
least-squares UnCut FEM on domains defined by a level set function\thanks{This research was partially supported by the National Natural Science Foundation of China (Grant No. 12525111) and by the AMSS--PolyU Joint Laboratory. Email address: buyang.li@polyu.edu.hk}
}

\author{
Jiashun Hu 
\qquad
Buyang Li 
\qquad
Han Yang 
\\[1.2em]
Department of Applied Mathematics, 
The Hong Kong Polytechnic University
}

\date{}

\hypersetup{
  pdftitle={
    An H-inverse least-squares UnCut FEM
    on domains defined by a level set function
  },
  pdfauthor={Jiashun Hu, Buyang Li, and Han Yang}
}

\begin{document}
\maketitle

\begin{abstract}
We propose a novel UnCut finite element method (FEM) for the Poisson and Stokes equations on domains with curved boundaries represented by a level set function. Like the $\phi$-FEM, the method avoids numerical integration over cut subregions of boundary elements, but introduces a novel least-squares formulation that minimizes an $H^{-1}$-residual of the governing equations. This formulation ensures stability without requiring large stabilization parameters, thereby eliminating the need for user-tuned penalty parameters and improving the robustness of the computation. Optimal-order convergence of the UnCut FEM solutions is rigorously established in the $H^1$ norm for both the Poisson and Stokes equations, and numerical experiments are presented to support the theoretical analysis.
\end{abstract}

\vspace{1.5em}

\section{Introduction}
This article is concerned with the numerical solutions of the Poisson and Stokes problems, namely, 
\begin{subequations}\label{Poisson}
\begin{align}
-\Delta u &= f \qquad\text{in } \Omega,\\
\label{2.1-1}
u &= 0 \qquad\text{on } \partial\Omega,
\end{align}
\end{subequations}
and
\begin{subequations}\label{Stokes}
\begin{align}
-2\operatorname{div} D u + \nabla p &= f \qquad\text{in } \Omega,\\
\nabla\!\cdot u &= 0 \qquad\text{in } \Omega,\\
u &= 0 \qquad \text{on } \partial\Omega,\label{3.1-3}
\end{align}
\end{subequations}
with \( Du := (\nabla u + \nabla u^{\rm T})/2 \), on a smooth domain $\Omega\subset\R^d$ implicitly described by a level set function $\phi:\R^d\rightarrow\R$, i.e., 
$\Omega=\{x\in\R^d:\phi(x)<0\}.$ 
Constructing fitted meshes for such domains is often computationally demanding. This difficulty has led to extensive research on unfitted finite element methods (FEMs) that do not require the computational mesh to fit the boundary. These methods employ a fixed background mesh, making them particularly attractive for problems involving intricate geometries or moving interfaces. 

Unfitted finite element methods in the discontinuous Galerkin framework were first introduced in \cite{HH2002CMAME} and have since attracted considerable attention. 
Over the past two decades, a wide range of unfitted finite element methods have been developed, most notably CutFEM
\cite{BCH+2015IJNME,BVM2018CMAME,BHL2022NM,BHL2023MC,BHL2024SJNAa,BHL+2025AN} and XFEM \cite{XFEM1,XFEM2,XFEM4}. 
For a comprehensive overview of recent developments, we refer to \cite{BHL+2025AN}.
A central challenge in these methods is the so-called small cut cell problem, which arises from arbitrarily small or highly anisotropic intersections between the boundary and background mesh elements.
To address this issue, various stabilization strategies have been proposed, including ghost penalty techniques \cite{MLL+2014JSC,BCH+2015IJNME,GSM2020SJSC} and cell-aggregation or merging approaches \cite{CLX2021NM,CL2023JCP,JL2013NM,CL2024ANM}, in which small cut cells are combined with neighboring elements to handle conditioning issues for higher-order elements.
Beyond steady-state problems, unfitted formulations have also been extended to time-dependent parabolic and Stokes equations
\cite{LO2019EM,MZ2022JCP,MZZ2022SJNA,OW2025CMAM} and fluid-structure interactions \cite{BURMAN2014497,FL2019IMA}.
For Stokes problems, unfitted discretizations involve additional theoretical challenges beyond the small cut cell problem. A fundamental difficulty is establishing the inf-sup stability condition on geometrically unfitted meshes, which has been addressed in \cite{GO2017MC,NO2024IJNA}. Another important consideration is the preservation of the divergence-free constraint.
In \cite{LNO2023EM}, the Scott--Vogelius element pair within the CutFEM framework
was shown to yield a nearly divergence-free velocity field.

Despite their flexibility and robustness, many unfitted FEMs remain challenging in practice. In particular, they require the accurate evaluation of integrals over the physical domain $\Omega$, which involves integration on curved cut cells, as illustrated in Figure~\ref{geo1}.
This typically necessitates specialized quadrature rules, which complicates the implementation of these methods.

\def\drawmeshlines{
    \draw[step=1, black!40, thick] (0,0) grid (\nx,\ny);
    \foreach \x in {0,...,7} {
        \foreach \y in {0,...,6} {
            \draw[black!40, thin] (\x,\y) -- (\x+1,\y+1);
        }
    }
}
\begin{figure}[!htbp]
  \centering
\begin{tikzpicture}[x=0.5cm,y=0.5cm]
\def\nx{8}
\def\ny{7}
\coordinate (C) at (4,3.5);
\def\a{2.6}
\def\b{1.9}
\def\cutcells{
  2/2,2/3,2/4,
  3/1,3/2,3/4,3/5,
  4/1,4/5,
  5/1,5/5,
  6/1,6/2,6/4,6/5,
  7/2,7/3,7/4
}
\def\fullcells{
  2/2,2/3,2/4,
  3/1,3/2,3/3,3/4,3/5,
  4/1,4/2,4/3,4/4,4/5,
  5/1,5/2,5/3,5/4,5/5,
  6/1,6/2,6/3,6/4,6/5,
  7/2,7/3,7/4
}

\def\cuttriangles{
  1/2/ 2/2/ 2/3,  1/2/ 1/3/ 2/3,
  1/3/ 2/3/ 2/4,  1/3/ 1/4/ 2/4,
  1/4/ 2/4/ 2/5, 
  2/1/ 3/1/ 3/2,  2/1/ 2/2/ 3/2,
  2/2/ 3/2/ 3/3,  2/2/ 2/3/ 3/3,
  2/4/ 2/5/ 3/5,
  2/5/ 3/5/ 3/6, 
  3/1/ 4/1/ 4/2,  3/1/ 3/2/ 4/2,
  3/5/ 4/5/ 4/6,  3/5/ 3/6/ 4/6,
  4/1/ 5/1/ 5/2,  4/1/ 4/2/ 5/2,
  4/5/ 5/5/ 5/6,  4/5/ 4/6/ 5/6,
  5/1/ 5/2/ 6/2,
  5/2/ 6/2/ 6/3,
  5/4/ 6/4/ 6/5,  5/4/ 5/5/ 6/5,
  5/5/ 6/5/ 6/6,  5/5/ 5/6/ 6/6,
  6/2/ 6/3/ 7/3,
  6/3/ 7/3/ 7/4,  6/3/ 6/4/ 7/4,
  6/4/ 7/4/ 7/5,  6/4/ 6/5/ 7/5
}

\begin{scope}
  \drawmeshlines
   \foreach \ax/\ay/\bx/\by/\cx/\cy in \cuttriangles {
   \begin{scope}
   \clip (\ax,\ay) -- (\bx,\by) -- (\cx,\cy) -- cycle;
      \path[fill=cyan!45,draw=none,even odd rule]
         (\ax,\ay) -- (\bx,\by) -- (\cx,\cy) -- cycle
        (C) ellipse [x radius=\a, y radius=\b];
  \end{scope}
  
  \begin{scope}
  \clip (C) ellipse [x radius=\a, y radius=\b];
  \clip (\ax,\ay) -- (\bx,\by) -- (\cx,\cy) -- cycle;
  \fill[pattern=north west lines, pattern color=black!70]
    (\ax,\ay) -- (\bx,\by) -- (\cx,\cy) -- cycle;
\end{scope}
  }
  \draw[very thick] (C) ellipse [x radius=\a, y radius=\b];
\end{scope}

\begin{scope}[shift={(10.3,0)}] 
  \foreach \ax/\ay/\bx/\by/\cx/\cy in \cuttriangles {
      \path[pattern=north east lines, pattern color=red] 
           (\ax,\ay) -- (\bx,\by) -- (\cx,\cy) -- cycle;
  }
  \drawmeshlines
\draw[blue, thick,line width=1.5pt]
(1,2) -- (2,3)
  (1,3) -- (2,3)
  (1,3) -- (2,4)
  (1,4) -- (2,4)
  (2,1) -- (3,2)
  (2,2) -- (2,3)
  (2,2) -- (3,2)
  (2,2) -- (3,3)
  (2,3) -- (2,4)
  (2,3) -- (3,3)
  (2,4) -- (2,5)
  (2,4) -- (3,5)
  (2,5) -- (3,5)
  (3,1) -- (3,2)
  (3,1) -- (4,2)
  (3,2) -- (3,3)
  (3,2) -- (4,2)
  (3,5) -- (3,6)
  (3,5) -- (4,5)
  (3,5) -- (4,6)
  (4,1) -- (4,2)
  (4,1) -- (5,2)
  (4,2) -- (5,2)
  (4,5) -- (4,6)
  (4,5) -- (5,5)
  (4,5) -- (5,6)
  (5,1) -- (5,2)
  (5,2) -- (6,2)
  (5,2) -- (6,3)
  (5,4) -- (5,5)
  (5,4) -- (6,4)
  (5,4) -- (6,5)
  (5,5) -- (5,6)
  (5,5) -- (6,5)
  (5,5) -- (6,6)
  (6,2) -- (6,3)
  (6,3) -- (6,4)
  (6,3) -- (7,3)
  (6,3) -- (7,4)
  (6,4) -- (6,5)
  (6,4) -- (7,4)
  (6,4) -- (7,5)
;
\draw[orange, thick,line width=1.5pt]
  (1,2) -- (1,3)
  (1,2) -- (2,2)
  (1,3) -- (1,4)
  (1,4) -- (2,5)
  (2,1) -- (2,2)
  (2,1) -- (3,1)
  (2,5) -- (3,6)
  (3,1) -- (4,1)
  (3,6) -- (4,6)
  (4,1) -- (5,1)
  (4,6) -- (5,6)
  (5,1) -- (6,2)
  (5,6) -- (6,6)
  (6,2) -- (7,3)
  (6,5) -- (6,6)
  (6,5) -- (7,5)
  (7,3) -- (7,4)
  (7,4) -- (7,5);
\end{scope}
\begin{scope}[shift={(20,0.5)}]
  \draw[cyan!45,fill=cyan]
    (0,4.2) rectangle (0.5,4.7);
  \node[anchor=west] at (0.7,4.45) {$B_h$};
  \path[
    pattern=north east lines,
    pattern color=red!70
  ] (0,3.2) rectangle (0.5,3.7);
  \draw[red]
    (0,3.2) rectangle (0.5,3.7);
      \node[anchor=west] at (0.7,3.45) {$\Omega_h^\Gamma,\mathcal{T}_h^\Gamma$};
      \begin{scope}[shift={(0,2)}]
  \path[
    pattern=north west lines,
    pattern color=black!70
  ] (0,3.2) rectangle (0.5,3.7);
  \draw
    (0,3.2) rectangle (0.5,3.7);
  \node[anchor=west] at (0.7,3.45) {cut cell};
	\end{scope}
	
\draw[orange, ultra thick]
  (0,2.2) -- (0.5,2.2);
\node[anchor=west] at (0.7,2.2) {$\partial\Omega_h$};
\draw[blue, ultra thick]
  (0,1.2) -- (0.5,1.2);
\node[anchor=west] at (0.7,1.2) {$ \mathcal{F}_h^\Gamma$};
\end{scope}
\end{tikzpicture}
\caption{Illustration of the extended domain $ \Omega_h$, cut cells in $ \Omega$, boundary strip $B_h:=\Omega_h\backslash\{\phi_h<0\}$, full cut element region $ \Omega_h^ \Gamma$ and the edge set $ {\mathcal{F}} _h^ \Gamma$ involving jump penalties.}
\label{geo1}
\end{figure}
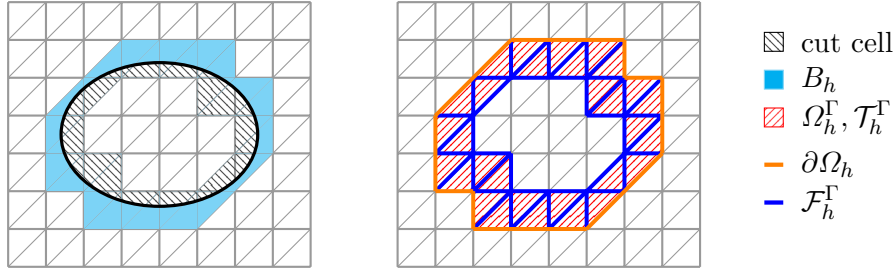

\FloatBarrier
The recently proposed $\phi$-FEM \cite{phiFEM_poisson,phifem_natural,phifem_heat,phifem_stokes} is an UnCut FEM which offers an appealing alternative that overcomes these integration challenges. In the $\phi$-FEM, the domain of the differential equation is extended to $ \Omega_h$, which consists of all full elements $T\in\mathcal{T}_h$
  intersecting the physical domain (including the cut cells $T\in\mathcal{T}_h^\Gamma$), and the Dirichlet boundary condition is implicitly imposed by constructing the approximate solution as the product of a finite element function and the level set function. Namely, one considers
\begin{align}\label{def:Sh}
S_h=\{\phi_hv_h:v_h\in H^1(\Omega_h)\ \text{and }v_h|_T\in P^k(T)\ \forall T\in\mathcal{T}_h\},
\end{align}
where $ \phi_h$ denotes an approximate level set function, and one introduces the operator $\tDelta: S_h\to S_h$ defined by
\begin{align}\label{def:tildedelta}
(\tDelta u_h,v_h)_{L^2(\Omega_h)}
=-(\nabla u_h,\nabla v_h)_{L^2(\Omega_h)}
+(\partial_n u_h,v_h)_{L^2(\partial \Omega_h)},\quad\forall\ v_h\in S_h,
\end{align}
and solves $-\tDelta u_h = P_h f$, where $P_h$ is the $L^2$ projection defined in \eqref{def:Ph}.
As a result, integrations over curved boundaries and cut cells are completely avoided, and all volume integrals are carried out on full background elements. This feature distinguishes $\phi$-FEM from other unfitted approaches, such as XFEM and CutFEM, and leads to substantial advantages in computational efficiency and implementation simplicity. To ensure stability in the presence of domain extension, $\phi$-FEM introduces stabilization terms on the elements $T\in \mathcal{T}_h^ \Gamma$ (which intersect the boundary), including volumetric penalties such as $\|\Delta u_h\|_T$ and edge-based jump stabilizations. For sufficiently large, though a priori unknown, stabilization parameters, this approach can be shown to achieve optimal-order convergence in the $H^1$ norm for both the Poisson and Stokes equations. However, in practical computations, the precise minimal values required for these stabilization parameters are generally not known. 

Motivated by this difficulty in practical computation, we propose a new unfitted finite element formulation, referred to as the \textit{least-squares UnCut~FEM}. The method is inspired by the ideas of the~$\phi$-FEM and employs the level set function~$\phi$ in its numerical construction. It avoids numerical integration over cut subregions of boundary elements as in the $\phi$-FEM, but introduces a novel least-squares formulation that minimizes an $H^{-1}$-residual of the governing equations, thereby enabling the use of $C^0$ finite elements and ensuring stability without requiring the stabilization parameters to be sufficiently large. As a result, the scheme remains robust with fixed unit weights, thereby eliminating the need for user-tuned stabilization parameters. For the Poisson equation, the proposed numerical scheme in~\eqref{2.7} is equivalent to the following minimization problem:
\begin{align}\label{mini:pro}\notag
    u_h=\arg\min_{v_h\in S_h}\Big(
    &\|(I-\Delta_h)^{-1}(\tDelta v_h+P_hf)\|_{H^1(\Omega_h)}^2\\
&+h^2\sum_{T\in\mathcal{T}_h^\Gamma}\|\Delta v_h+f\|_{L^2(T)}^2
        +j_h(v_h,v_h)\Big),
\end{align}
where $j_h(\cdot,\cdot)$ is a penalty term that ensures stability of the numerical scheme; see \eqref{penalty:edge} for its definition.
In particular, the first term in the functional above 
represents the squared discrete 
$H^{-1}(\Omega_h)$-norm of the residual
$\tDelta u_h+P_h f$, realized via the discrete Riesz map $(I-\Delta_h)^{-1}$,
with $\Delta_h:S_h\to S_h$ the discrete Laplacian operator with homogeneous Neumann boundary conditions defined in \eqref{def:deltah}.

Similarly, for the Stokes problem, the proposed numerical scheme is equivalent to the following minimization problem:
\begin{align}\label{functional:stokes}\notag
(u_h,p_h)=&\arg\min_{v_h\in X_h,q_h\in Q_h}
\Big(\|(I-\Delta_h)^{-1} (K(v_h,q_h)+P_hf)\|_{H^1(\Omega_h)}^2+\|\nabla\cdot v_h\|_{L^2(\Omega_h)}^2 \\
&+h^2\sum_{T\in\mathcal{T}_h^\Gamma}\|-2 \textrm{div}D v_h+\nabla q_h-f\|_{L^2(T)}^2+j_h^{(2)}(v_h,v_h)\Big),
\end{align}
where the finite element spaces $X_h = (S_h)^d$ and $Q_h$, as well as the penalty term $j_h^{(2)}$, are defined in \eqref{defj2}.
Similar to \eqref{def:tildedelta}, $K(u_h,p_h)\in X_h$ is defined by requiring for $v_h\in X_h$,
\begin{align}\notag\label{defK}
    (K(u_h,p_h),v_h)_{ L^2(\Omega_h)}&= -(2D u_h, D v_h)_{ L^2(\Omega_h)}
    + (p_h, \nabla \cdot v_h)_{L^2( \Omega_h)}\\
    &+ ((2D u_h - p_hI) n, v_h)_{L^2(\partial \Omega_h)}.
\end{align}
Introducing $w_h\in X_h$ as the Riesz representative of this residual, i.e.,
\begin{equation}\label{eq:residue}
(I-\Delta_h) w_h = K(u_h,p_h)+P_h f,
\end{equation}
the minimization problem can be reformulated in terms of $w_h$.
The resulting functional is coercive with respect to
$w_h$, $\nabla\!\cdot u_h$, and the additional stabilization terms.
As a consequence, it directly yields stability estimates for
$\|w_h\|_{H^1(\Omega_h)}$, $\|\nabla\!\cdot u_h\|_{L^2(\Omega_h)}$, and the
penalty contributions.

This is precisely the reason why the proposed method succeeds. 
Taking the Stokes problem as an example, the stabilization is designed 
to be independent of the velocity--pressure coupling mechanism. 
In particular, once the stability of $w_h$, $\nabla\!\cdot u_h$, 
and the additional penalty terms has been established, the stability of the discrete velocity $u_h$ and pressure $p_h$ can then be deduced from the residual relation~\eqref{eq:residue}. 
In other words, the intrinsic stability structure of the least-squares UnCut FEM allows the stabilization terms to be strong enough in the proof without adversely affecting the pressure stability in the Stokes problem. Indeed, although the numerical scheme is revisited after applying the inf--sup condition, the resulting estimate for $\|p_h\|_{L^2(\Omega_h)}$ relies on the previously established bounds for $w_h$ and penalty terms. 
As a consequence, optimal $H^1$-error estimates for the velocity and optimal
$L^2$-error estimates for the pressure (up to a constant) can be derived.
Moreover, the proposed approach yields a \emph{parameter-free} scheme that is robust with respect to the cut configuration, as demonstrated numerically (i.e., it maintains accuracy even when mesh elements are arbitrarily small or irregularly cut).

The remainder of this paper is organized as follows. In Section 2, we introduce the notation and assumptions, present the least-squares UnCut FEM schemes
for the Poisson and Stokes equations, and state our main results. In Section 3, we give detailed proofs of the results stated in Section 2. In Section 4, we present numerical experiments that validate our theoretical results. Finally, concluding remarks are presented in Section 5.

\section{Formulation of the Method and Main Results}
\subsection{Notation and Assumptions} 
In this paper, $\mathcal{O}\subset\mathbb{R}^d$ ($d=2,3$) denotes a bounded background domain with a quasi-uniform simplicial triangulation $\mathcal{T}_h^{\mathcal{O}}$ of mesh size $h$.
Let $(\phi_h)_{h>0}$ be a family of continuous, elementwise
polynomial approximations of the smooth level set function $\phi$,
with uniformly bounded polynomial degree (e.g., nodal interpolation onto the Lagrange finite element space on $\mathcal{T}_h^{\mathcal{O}}$). 
Based on $\phi_h$, we define the \emph{active triangulation}
as the collection of elements in the background mesh $\mathcal{T}_h^{\mathcal{O}}$ that intersect
the interior $\{\phi_h<0\}$:
\[
\mathcal{T}_h
=\bigl\{\,T\in\mathcal{T}_h^{\mathcal{O}}:\; T\cap\{\phi_h<0\}\neq\emptyset\,\bigr\}.
\]
The domain occupied by the active triangulation is then denoted as
$\Omega_h:=\bigl(\bigcup_{T\in\mathcal{T}_h}\overline{T}\bigr)^{o}.$
We next single out those elements that are cut by the approximate interface $\{\phi_h=0\}$ and the
corresponding internal facets adjacent to such cut elements (see Figure~\ref{geo1}):
\begin{align}
\mathcal{T}_h^\Gamma
&=\bigl\{\,T\in\mathcal{T}_h:\; T\cap\Gamma_h\neq\emptyset\,\bigr\},\quad \Gamma_h:=\{\phi_h=0\},
\\
\mathcal{F}_h^\Gamma
&=\bigl\{\,F\text{ is a common facet }((d-1)\text{face})\text{ such that }F=\overline{T_1}\cap\overline{T_2}:\; T_1\in\mathcal{T}_h,\; T_2\in\mathcal{T}_h^\Gamma\,\bigr\}.
\end{align}
The union of cut elements defines the corresponding interface zone
$\Omega_h^\Gamma:=\bigl(\bigcup_{T\in\mathcal{T}_h^\Gamma}\overline{T}\bigr)^{o}.$

Throughout this paper, we use $(\cdot,\cdot)_{L^2}$ and $(\cdot,\cdot)_{H^1}$ to denote the $L^2$ and $H^1$ inner products, respectively. We write $C$ for a generic positive constant independent of $h$, whose value may change from line to line. We also impose the following mild assumptions on $\phi$, $\phi_h$, and $\mathcal{T}_h^\Gamma$.
\begin{assumption}\label{a1}
We assume that $\phi$ is smooth with $|\nabla \phi(x)| > 0$ for all $x \in \Gamma$ and $\phi(x) \neq 0$ for all $x \notin \Gamma$. 
\end{assumption}

In practical computation, if $\phi_h$ is a good approximation to $\phi$ (such as the Lagrange interpolant of $\phi$), then Assumption~\ref{a1} implies the following properties (which we state as assumptions) when the mesh size $h$ is sufficiently small.

\begin{assumption}\label{a3}
For every cut element $T\in\mathcal{T}_h^\Gamma$, there exists a path connecting $T$ and an interior element $T'\in\mathcal{T}_h\backslash\mathcal{T}_h^\Gamma$ that goes through at most $N$ cut elements in $\mathcal{T}_h^\Gamma$, where $N\in\mathbb{N}_{>0}$ is independent of the mesh size $h$.
\end{assumption}
\begin{assumption}\label{a2}
We assume that $|\nabla \phi_h(x)|$ is uniformly bounded  and there exists an $h$-independent constant $m>0$ such that $|\nabla \phi_h(x)|>m$ whenever $\mathrm{dist}(x, \Gamma) < 3h$. 
Moreover, we assume that  $\mathrm{dist}(\Gamma, \Gamma_h) = O(h)$.
\end{assumption}

\begin{myremark}\upshape
Assumption~\ref{a1} is explicitly used in the proof of Lemma~\ref{interpolation}. 
Assumption~\ref{a3} is used in the proofs of Lemmas~\ref{poissonlemma} and~\ref{phifem_stokes}, while Assumption~\ref{a2} is required for Lemma~\ref{oldinfsup}, following the analysis developed in the $\phi$-FEM literature~\cite{phiFEM_poisson,phifem_natural,phifem_heat,phifem_stokes}. 
\end{myremark}

\subsection{Numerical Schemes and Main Results}
\subsubsection{Poisson Equation}
We first consider the Poisson equation \eqref{Poisson} with homogeneous Dirichlet boundary conditions.
To formulate the weak form of the minimizer of \eqref{mini:pro}, we introduce an auxiliary variable $w_h\in S_h$ defined by
\begin{align}
\label{def:wh}
w_h=(I-\Delta_h)^{-1}(\tDelta u_h+ P_h f),
\end{align} 
where $\tDelta:S_h\to S_h$ is defined in \eqref{def:tildedelta}, 
and $\Delta_h:S_h\to S_h$ is the discrete Laplacian operator with homogeneous Neumann boundary conditions on $\partial\Omega_h$, i.e.
\begin{align}\label{def:deltah}
(\Delta_h s_h,v_h)_{L^2(\Omega_h)}
:=-(\nabla s_h,\nabla v_h)_{L^2(\Omega_h)},\quad\forall v_h\in S_h.
\end{align}
Moreover, $P_h$ is the $L^2$ projection operator given by
\begin{align}\label{def:Ph}
(P_hf,v_h)_{L^2(\Omega_h)}
=(f,v_h)_{L^2(\Omega_h)},\quad\forall v_h\in S_h.
\end{align}
Consequently, \eqref{def:wh} can be written in the following weak form: find $w_h\in S_h$ such that
\begin{align}\label{scheme2}
(w_h,\eta_h)_{L^2(\Omega_h)}
+(\nabla w_h,\nabla\eta_h)_{L^2(\Omega_h)}
= (\tilde \Delta_h u_h+ P_h f,\eta_h)_{L^2(\Omega_h)}
\qquad \forall\,\eta_h\in S_h.
\end{align}
Note that $f$ is evaluated on $\Omega_h$ above. For simplicity, we assume that $f$ is defined on $\Omega_h\cup\Omega$, so that its restriction to $\Omega_h$ is well-defined.

By introducing the auxiliary variable, the variational formulation of (\ref{mini:pro}) reads: for all $v_h\in S_h$, it holds that
\begin{equation}\label{scheme3}
\big(w_h,(I-\Delta_h)^{-1} \tilde\Delta_h v_h\big)_{H^1(\Omega_h) }+h^2\sum_{T\in\mathcal{T}_h^\Gamma}(\Delta u_h+f, \Delta v_h)_{L^2(T)}+j_h(u_h,v_h)=0,
\end{equation}
where
\begin{align}\label{penalty:edge}
j_h(u_h,v_h)&=\sum_{E\in\mathcal{F}_h^\Gamma}h\int_{E}[\partial_nu_h]_E[\partial_n v_h]_E.
\end{align}
Note that for $z_h\in S_h$, we have
\begin{align}\label{prop:delta}
(w_h,z_h)_{H^1(\Omega_h)}=(w_h,z_h)_{L^2(\Omega_h)}+(\nabla w_h,\nabla z_h)_{L^2(\Omega_h)}=(w_h,(I-\Delta_h)z_h)_{L^2(\Omega_h)}.
\end{align}
This allows us to rewrite (\ref{scheme3}) as
\begin{equation}\label{scheme4}
(w_h,\tDelta v_h)_{L^2(\Omega_h)}
+h^2\sum_{T\in\mathcal{T}_h^\Gamma}(\Delta u_h+f,\Delta v_h)_{L^2(T)}
+j_h(u_h,v_h)=0.
\end{equation}
Combining (\ref{scheme2}) with (\ref{scheme4}), our scheme reads as follows: find $(w_h,u_h)\in S_h\times S_h$ such that, for all $(\eta_h,v_h)\in S_h\times S_h$,
it holds that
\begin{align}\label{2.7}\notag
&(w_h,\eta_h)_{H^1(\Omega_h)}
-(\partial_n u_h,\eta_h)_{L^2(\partial\Omega_h)}
+(\nabla u_h,\nabla\eta_h)_{L^2(\Omega_h)}
=(f,\eta_h)_{L^2(\Omega_h)},\\
\notag
&(w_h,\partial_n v_h)_{L^2(\partial\Omega_h)}
-(\nabla w_h,\nabla v_h)_{L^2(\Omega_h)}
+h^2\sum_{T\in\mathcal{T}_h^\Gamma}(\Delta u_h,\Delta v_h)_{L^2(T)}\\
&\hspace{6cm}+j_h(u_h,v_h)=-h^2\sum_{T\in\mathcal{T}_h^\Gamma}(f,\Delta v_h)_{L^2(T)} .
\end{align}
We now state our main result for solving the Poisson equation \eqref{Poisson}:
\begin{theorem}\label{thm2.1}
Let $f\in H^{k-1}(\Omega_h\cup\Omega) \cap H^1(\Omega_h\cup\Omega)$,
and let $\tilde{u}\in H^{k+1}(\Omega_h\cup \Omega_)\cap H^{3}(\Omega_h\cup \Omega)$ be an extension of the solution $u$.
Then, under Assumptions~\ref{a1}--\ref{a3}, there exists a unique pair $(w_h,u_h)\in S_h\times S_h$ satisfying (\ref{2.7}), and it holds that 
\begin{equation}\label{thm:poisson_u}
\|\tilde{u}-u_h\|_{H^1(\Omega_h)}\le C\big(h^k+\|\phi-\phi_h\|_{H^1(\Omega_h)}\big)
(\|f\|_{H^{k-1}(\Omega_h\cup\Omega)} + \|f\|_{H^1(\Omega_h\cup\Omega)}). 
\end{equation}
\end{theorem}
\subsubsection{Stokes Equations}
We now consider the Stokes problem \eqref{Stokes} subject to homogeneous Dirichlet boundary conditions.
The Taylor--Hood finite element spaces for the velocity and pressure are defined as follows:
\begin{align}\label{defX_H}
X_h&= (S_h)^d,\quad S_h=\{\phi_hv_h:v_h\in H^1(\Omega_h)\ \text{and }v_h|_T\in P^k(T)\ \forall\ T\in\mathcal{T}_h\},\\
\label{defQ_H}
Q_h&=\{q_h\in H^1(\Omega_h):q_h|_{T}\in P^{k-1}(T)\ \forall\ T\in\mathcal{T}_h,\text{ and }\int_{\Omega_h}q_h dx=0\}.
\end{align}
Similar to the least-squares formulation \eqref{2.7} for the Poisson problem, 
we introduce an auxiliary variable 
$w_h\in X_h$ defined by
\begin{equation}
w_h=(I-\Delta_h)^{-1}(K(u_h,p_h)+P_h f),
\end{equation}
where, by a slight abuse of notation, we extend the discrete Laplacian $\Delta_h$ to vector fields by applying it to each component.
Then, in weak form, for all $\eta_h\in X_h$, we have
\begin{align}\label{scheme2.2}\notag
(w_h,\eta_h)_{H^1(\Omega_h)}
&=-2(Du_h,D\eta_h)_{L^2(\Omega_h)}+(p_h,\nabla\cdot\eta_h)_{L^2(\Omega_h)}\\
&+((2Du_h-p_hI)n,\eta_h)_{L^2(\partial\Omega_h)}
+(f,\eta_h)_{L^2(\Omega_h)}.
\end{align}
Using \eqref{prop:delta}, the variational formulation of \eqref{functional:stokes} reads as follows: 
\begin{align}\label{scheme2.4}\notag
&\!-2(Dw_h,Dv_h)_{L^2(\Omega_h)}
\!+\!(\nabla\!\cdot\! w_h,q_h)_{L^2(\Omega_h)}
\!+\!(w_h,(2Dv_h \!-\!q_hI)n)_{L^2(\partial\Omega_h)}
\!+\!(\nabla\!\cdot\! u_h,\nabla \!\cdot\! v_h)_{L^2(\Omega_h)}\\
&\!+\!h^2\sum_{T\in\mathcal{T}_h^\Gamma}(-2\textrm{div}D u_h
\!+\!\nabla p_h-f,-2 \textrm{div}D v_h+\nabla q_h)_{L^2(T)}
+j_h^{(2)}(u_h,v_h)=0,
\end{align}
where
\begin{align}\label{defj2}
j_h^{(2)}(u_h,v_h)&=\sum_{E\in\mathcal{F}_h^\Gamma}h\int_{E}[\partial_nu_h]_E[\partial_n v_h]_E+h^3\int_E[\partial_n^2u_h]_E[\partial_n^2 v_h]_E.
\end{align}
Hence, the scheme is obtained by combining \eqref{scheme2.2} and \eqref{scheme2.4}: 
find $(w_h,u_h,p_h)\in X_h\times X_h\times Q_h$ such that, 
for all $(\eta_h,v_h,q_h)\in X_h\times X_h\times Q_h$, it holds that
\begin{align}\label{3.7}\notag
&(w_h,\eta_h)_{H^1(\Omega_h)}
+(\nabla \cdot u_h,\nabla \cdot v_h)_{L^2(\Omega_h)}
\\\notag
&-((2Du_h-p_hI)n,\eta_h)_{L^2(\partial\Omega_h)}
+2(Du_h,D\eta_h)_{L^2(\Omega_h)}-(p_h,\nabla\cdot\eta_h)_{L^2(\Omega_h)}
\\\notag
&+(w_h,(2Dv_h-q_hI)n)_{L^2(\partial\Omega_h)}
-2(Dw_h,Dv_h)_{L^2(\Omega_h)}+(\nabla\cdot w_h,q_h)_{L^2(\Omega_h)}\\\notag
&+h^2\sum_{T\in\mathcal{T}_h^\Gamma}(-2 \textrm{div}D u_h+\nabla p_h,
-2 \textrm{div}D v_h+\nabla q_h)_{L^2(T)}
+j_h^{(2)}(u_h,v_h)\\
&{}=\ (f,\eta_h)_{L^2(\Omega_h)}
+h^2\sum_{T\in\mathcal{T}_h^\Gamma}(f,-2 \textrm{div}D v_h+\nabla q_h)_{L^2(T)}.
\end{align}
We now state the main results for the scheme \eqref{3.7}.
\begin{theorem}\label{thm:stokes}
Let $k\geq 2$. Let $f\in H^{k-1}(\Omega_h\cup\Omega)\cap H^1(\Omega_h\cup\Omega)$, 
and let $\tilde{u}\in H^{k+1}(\Omega_h \cup\Omega)\cap 
H^{3}(\Omega_h \cup\Omega)$ and $\tilde{p}\in H^{k}(\Omega_h\cup\Omega)$ be extensions of the solutions
$u$ and $p$, respectively.
Then, under Assumptions \ref{a1}--\ref{a2}, there exists a unique triple $(w_h,u_h,p_h)\in X_h\times X_h\times Q_h$ satisfying (\ref{3.7}), and it holds that
\begin{align}\label{thm:u}
\|u_h-\tilde{u}\|_{H^1(\Omega_h)}&\le C\big(h^k+\|\phi-\phi_h\|_{H^1(\Omega_h)}\big)(\|f\|_{H^{k-1}(\Omega_h\cup\Omega)}+\|f\|_{H^1(\Omega_h\cup\Omega)}),\\
\label{thm:p}
\|p_h - \Pi \tilde{p}\|_{L^2(\Omega_h)}&\le C\big(h^k+\|\phi-\phi_h\|_{H^1(\Omega_h)}\big)(\|f\|_{H^{k-1}(\Omega_h\cup\Omega)}+\|f\|_{H^1(\Omega_h\cup\Omega)}),
\end{align}
where $\Pi$ is defined by

\vspace{-1.5em}
\begin{equation}
\label{def:Pi}
\Pi g := g-\bar g, 
\qquad 
\bar g := \frac{1}{|\Omega_h|}\int_{\Omega_h} g\,dx .
\end{equation}
\end{theorem}

\vspace{-0.5em}
\begin{myremark}\upshape\label{rem:projp}
The error bounds in Theorem \ref{thm2.1}--\ref{thm:stokes} include the term $\|\phi-\phi_h\|_{H^1(\Omega_h)}$ on the right hand side. 
In particular, if $\phi_h$ is given by the degree-$k$ (Lagrange) finite element interpolant of $\phi$, then the error estimates in \eqref{thm:poisson_u} and 
\eqref{thm:u}--\eqref{thm:p} reduce to $O(h^k)$.
\end{myremark}


\section{Proof of Theorems~\ref{thm2.1}-\ref{thm:stokes}}    
In this section, we present the proofs of the main theorems of the paper. 
We begin by recalling several preliminary tools that are frequently used in the $\phi$-FEM framework, including 
Poincar\'e inequalities, trace inequalities, and Hardy's inequality.
For ease of reference, these results are presented using their original numbering and are stated without proof.

\begin{lemma}[{\!\!\cite[Lemma~3.6]{phiFEM_poisson}}]\label{vanish}
Suppose that Assumption~2.1 holds.
Then, for any $v\in H^s(\Omega_h)$ 
satisfying $v=0$ in $\Omega$, it holds that
\[
\|v\|_{L^2(\Omega_h\setminus\Omega)}\le C h^s \|v\|_{H^s(\Omega_h\setminus\Omega)}.
\]
\end{lemma}
\begin{lemma}[{\!\!\cite[Lemma 3.1]{phiFEM_poisson}}]\label{hardy}
For any $v\in H^{s+1}(\Omega_h)$ vanishing on $\partial\Omega$, we have
\[\|v/\phi\|_{H^s(\Omega_h)}\le C\|v\|_{H^{s+1}(\Omega_h)}.\]
\end{lemma}

\begin{lemma}[{\!\!\cite[Lemmas~3.4, 3.5]{phiFEM_poisson}}]\label{poincare1}
For $v_h\in S_h$, the following estimates hold:
\begin{align}\label{eq:poincare-inequ}
&\|v_h\|_{L^2(\Omega_h^\Gamma)}\le C h\,| v_h|_{H^1(\Omega_h^\Gamma)},\\
\label{eq:trace-inequ}
&\sum_{E\in\mathcal{F}_h^\Gamma}\| v_h\|_{L^2(E)}^2
 \le C h\,| v_h|_{H^1(\Omega_h^\Gamma)}^2
 \le C h^{-1}\| v_h\|_{L^2(\Omega_h^\Gamma)}^2,\\
&\| v_h\|_{L^2(\partial\Omega_h)}^2
 \le C h\,| v_h|_{H^1(\Omega_h^\Gamma)}^2
 \le C h^{-1}\| v_h\|_{L^2(\Omega_h^\Gamma)}^2.
\end{align}
\end{lemma}
For later use, we also recall the standard elementwise scaled trace
inequality. For every function $v$ such that $v|_T\in H^1(T)$ for all
$T\in\mathcal T_h^\Gamma$, we have
\[
\|v\|_{L^2(\partial\Omega_h)}^2
\leq
C\left(
h^{-1}\|v\|_{L^2(\Omega_h^\Gamma)}^2
+
h\sum_{T\in\mathcal T_h^\Gamma}|v|_{H^1(T)}^2
\right).
\tag{3.3a}
\]

\begin{lemma}[Poincar\'e inequality on $ \Omega_h$]\label{poincare2}
For all $v_h\in S_h$, it holds that
\[\|v_h\|_{L^2(\Omega_h)}\le C|v_h|_{H^1(\Omega_h)}.\]
\end{lemma}
\begin{proof}
Note that $v_h=0$ on $\{\phi_h=0\}$. 
By the standard Poincar\'e inequality, we obtain
\begin{equation}\label{pc}
\|v_h\|_{L^2(\{\phi_h<0\})}
\le C\,|v_h|_{H^1(\{\phi_h<0\})}.
\end{equation}
Hence, combined with Lemma~\ref{poincare1}, it follows that
\begin{align*}
\|v_h\|_{L^2(\Omega_h)}
&\le \|v_h\|_{L^2(\Omega_h^\Gamma)}
   + \|v_h\|_{L^2(\{\phi_h<0\})} 
   \le C\,|v_h|_{H^1(\Omega_h)}.
\end{align*}
\hfill\end{proof}

The following lemma establishes the interpolation error estimate, where the dependence on $\|\phi - \phi_h\|_{H^1(\Omega_h)}$ is explicitly stated.
\begin{lemma}
\label{interpolation}
There exists an operator
\[
\widetilde I_h:
\left\{
v\in H^{k+1}(\Omega_h)\cap H^3(\Omega_h):
v|_{\partial\Omega}=0
\right\}
\longrightarrow S_h
\]
such that
\[
\|v-\widetilde I_hv\|_{H^1(\Omega_h)}
\leq
C\bigl(h^k+\|\phi-\phi_h\|_{H^1(\Omega_h)}\bigr)
\left(
\|v\|_{H^{k+1}(\Omega_h)}
+\|v\|_{H^3(\Omega_h)}
\right).
\]
Moreover,
\[
h^{1/2}
\|\nabla(v-\widetilde I_hv)\|_{L^2(\partial\Omega_h)}
\leq
C\bigl(h^k+\|\phi-\phi_h\|_{H^1(\Omega_h)}\bigr)
\left(
\|v\|_{H^{k+1}(\Omega_h)}
+\|v\|_{H^3(\Omega_h)}
\right).
\tag{3.9a}
\]
\end{lemma}

\begin{proof}
Let $v\in H^{k+1}(\Omega_h)\cap H^3(\Omega_h)$ with $v=0$ on $\partial\Omega$, and set $z=v/\phi$. From Lemma \ref{hardy}, we have $z\in H^{k}(\Omega_h)$ and $\|z\|_{H^k(\Omega_h)}\le C\|v\|_{H^{k+1}(\Omega_h)}$. 
Let $z_h$ be the Scott-Zhang 
interpolation of $z$. Define $\tilde{I}_h v := \phi_h z_h$. The interpolation error can be decomposed as follows:
\begin{equation}
\label{eq:split}
\phi z - \phi_h z_h = \underbrace{(\phi-\phi_h)z}_{T_1} + \underbrace{\phi(z-z_h)}_{T_2} - \underbrace{(\phi-\phi_h)(z-z_h)}_{T_3}.
\end{equation}
We estimate the terms $T_1$, $T_2$, and $T_3$ in the $H^1(\Omega_h)$-norm separately. 
\begin{align*}
\|T_1\|_{H^1(\Omega_h)} 
&\le \|z\nabla(\phi-\phi_h)\|_{L^2(\Omega_h)} + \|(\phi-\phi_h)\nabla z\|_{L^2(\Omega_h)} + \|(\phi-\phi_h)z\|_{L^2(\Omega_h)} \\
&\le \|z\|_{L^\infty(\Omega_h)}\|\nabla(\phi-\phi_h)\|_{L^2(\Omega_h)} 
+ \|\phi-\phi_h\|_{L^6(\Omega_h)}\big(\|\nabla z\|_{L^3(\Omega_h)}+\|z\|_{L^3(\Omega_h)}\big) \\
&\le C\|z\|_{H^2(\Omega_h)}\|\phi-\phi_h\|_{H^1(\Omega_h)}.
\end{align*}
Further we have $\|z\|_{H^2(\Omega_h)}\le C
\|v\|_{H^{3}(\Omega_h)}$.
For the third term, we employ H\"older's inequality and Sobolev embeddings:
\begin{align*}
\|T_3\|_{H^1(\Omega_h)} 
&\le \|(z-z_h)\nabla(\phi-\phi_h)\|_{L^2(\Omega_h)} + \|(\phi-\phi_h)\nabla(z-z_h)\|_{L^2(\Omega_h)} \\
&\quad + \|(\phi-\phi_h)(z-z_h)\|_{L^2(\Omega_h)} \\
&\le \|z-z_h\|_{L^\infty(\Omega_h)}\|\phi-\phi_h\|_{H^1(\Omega_h)} 
+ \|z-z_h\|_{W^{1,3}(\Omega_h)}
\|\phi-\phi_h\|_{L^6(\Omega_h)}\\
&\le C\|z\|_{H^{2}(\Omega_h)}\|\phi-\phi_h\|_{H^1(\Omega_h)} \\
&\le C\|v\|_{H^{3}(\Omega_h)}\|\phi-\phi_h\|_{H^1(\Omega_h)}.
\end{align*}
Here, we utilized the standard Sobolev embeddings: $H^2(\Omega_h)\hookrightarrow L^\infty(\Omega_h)$ and $H^1(\Omega_h)\hookrightarrow L^6(\Omega_h)\hookrightarrow L^3(\Omega_h)$. 
For the second term, using the boundedness of $\phi$, we obtain
\begin{align*}
\|T_2\|_{H^1(\Omega_h)} 
&\le \|(z-z_h)\nabla\phi\|_{L^2(\Omega_h)} + \|\phi\nabla(z-z_h)\|_{L^2(\Omega_h)} + \|\phi(z-z_h)\|_{L^2(\Omega_h)} \\
&\le C h^k \|\phi\|_{W^{1,\infty}(\Omega_h)}\|v\|_{H^{k+1}(\Omega_h)} + \|\phi\nabla(z-z_h)\|_{L^2(\Omega_h)}
+\|\phi(z-z_h)\|_{L^2(\Omega_h)}. 
\end{align*}
To estimate the last two terms, we follow the approach of \cite[Lemma 6]{phifem_stokes} and distinguish between elements close to the boundary and elements away from it.
By Assumption~\ref{a1} and the compactness of $\Gamma$, 
$\nabla\phi$ is bounded and nonvanishing in a neighborhood of $\Gamma$.
Together with the fact that $\phi \neq 0$ for $x\not\in\Gamma$, there exists a constant $c_0>0$ such that, for sufficiently small $h$, 
\[
|\phi(x)| > c_0 h \qquad \text{for all } x \text{ satisfying } \mathrm{dist}(x,\Gamma) > 3h .
\]
Now let $\omega_T$ be the patch corresponding to the Scott-Zhang interpolation on element $T$.
We first consider the case $\mathrm{dist}(\omega_T,\Gamma)> 3h$,
in which we may invoke \cite[(33)]{phifem_stokes} to obtain the higher-derivative estimate on the patch $\omega_T$:
\begin{equation}
|z|_{H^{k+1}(\omega_T)}
\le \frac{1}{\min_{x\in \omega_T} |\phi(x)|}\Big( |v|_{H^{k+1}(\omega_T)} + C \|z\|_{H^{k}(\omega_T)}\Big).
\end{equation}
Consequently, it follows that
\begin{equation}
\|\phi \nabla(z-z_h)\|_{L^2(T)}
\le C h^k \frac{\|\phi\|_{L^\infty(T)}}{\min_{x\in \omega_T}|\phi(x)|}
\Big( |v|_{H^{k+1}(\omega_T)}+\|z\|_{H^{k}(\omega_T)}\Big).
\end{equation}
Then we can bound the ratio as
\begin{align}
\frac{\|\phi\|_{L^\infty(T)}}{\min_{x\in \omega_T}|\phi(x)|}
&= 1+\frac{\|\phi\|_{L^\infty(T)}-\min_{x\in \omega_T}|\phi(x)|}{\min_{x\in \omega_T}|\phi(x)|}
\le 1+\frac{2\|\nabla\phi\|_{L^\infty}}{c_0}.
\end{align}
Next, we consider $T$ with $\mathrm{dist}(\omega_T,\Gamma)<3h$. 
Since $\|\nabla\phi\|_{L^\infty}$ is bounded, we have
\begin{equation}
\|\phi \nabla(z-z_h)\|_{L^2(T)}
\le \|\phi\|_{L^\infty(T)}\,\|\nabla(z-z_h)\|_{L^2(\omega_T)}
\le C\, h^k \|z\|_{H^{k}(\omega_T)} .
\end{equation}
By summing the above estimates over all elements $T$, we obtain
\begin{align*}
\|T_2\|_{H^1(\Omega_h)}
\le C h^k \bigl(1+\|\phi\|_{W^{1,\infty}(\Omega_h)}\bigr)\,\|v\|_{H^{k+1}(\Omega_h)}.
\end{align*}
To prove (3.9a), let $I_hv$ denote the standard degree-$k$
interpolant of $v$. We split
\[
v-\widetilde I_hv
=
(v-I_hv)+(I_hv-\widetilde I_hv).
\]

For the first term, applying the scaled trace inequality (3.3a)
elementwise to $\nabla(v-I_hv)$ and using the standard interpolation
estimates, we obtain
\[
\begin{aligned}
h^{1/2}
\|\nabla(v-I_hv)\|_{L^2(\partial\Omega_h)}
&\leq
C\left(
\|\nabla(v-I_hv)\|_{L^2(\Omega_h^\Gamma)}
+
h\left(
\sum_{T\in\mathcal T_h^\Gamma}
|v-I_hv|_{H^2(T)}^2
\right)^{1/2}
\right)                                                       \\
&\leq
Ch^k\|v\|_{H^{k+1}(\Omega_h)}.
\end{aligned}
\]

For the second term, $I_hv-\widetilde I_hv$ is an elementwise
polynomial of uniformly bounded degree. Hence, the inverse trace
inequality gives
\[
\begin{aligned}
h^{1/2}
\|\nabla(I_hv-\widetilde I_hv)\|_{L^2(\partial\Omega_h)}
&\leq
C\|\nabla(I_hv-\widetilde I_hv)\|_{L^2(\Omega_h^\Gamma)}       \\
&\leq
C\left(
\|\nabla(v-I_hv)\|_{L^2(\Omega_h)}
+
\|\nabla(v-\widetilde I_hv)\|_{L^2(\Omega_h)}
\right)                                                       \\
&\leq
C\bigl(h^k+\|\phi-\phi_h\|_{H^1(\Omega_h)}\bigr)
\left(
\|v\|_{H^{k+1}(\Omega_h)}
+\|v\|_{H^3(\Omega_h)}
\right).
\end{aligned}
\]
Combining the two estimates proves (3.9a).
\hfill\end{proof}

Next, we recall  \cite[Lemma~3.7]{phiFEM_poisson}, which establishes a crucial coercivity property once a sufficiently strong penalty term is added.
\begin{lemma}\label{poissonlemma}
Under Assumptions~\ref{a1}--\ref{a3}, there exists a constant 
$C_{\mathrm{coer}}\in (0,1)$, independent of $h$, such that, 
provided $\sigma>0$ is sufficiently large, we have
\begin{align}\label{c031}\notag
&(\nabla u_h,\nabla u_h)_{L^2(\Omega_h)}-(\partial_n u_h,u_h)_{L^2(\partial\Omega_h)}
+\sigma h^2\sum_{T\in\mathcal{T}_h^\Gamma}
\|\Delta u_h\|_{L^2(T)}^2
+\sigma j_h(u_h,u_h)\\
&\quad \ge C_{\mathrm{coer}}\Big((\nabla u_h,\nabla u_h)_{L^2(\Omega_h)}
+ h^2\sum_{T\in\mathcal{T}_h^\Gamma}\|\Delta u_h\|_{L^2(T)}^2
+ j_h(u_h,u_h)
\Big).
\end{align}
\end{lemma}
\subsection{Proof of Theorem~\ref{thm2.1}}
The proof of Theorem~\ref{thm2.1} is divided into two parts. 
We first establish the stability of the proposed scheme, and then verify its consistency.

\subsubsection{Stability}
We rewrite \eqref{2.7} in the following abstract form:
\begin{equation}\label{2.10}
a_h(w_h,u_h;\eta_h,v_h)=\ell(\eta_h,v_h),
\end{equation}
where the bilinear and linear forms are defined as follows
\begin{align*}
a_h(w_h,u_h;\eta_h,v_h):=&{}
(w_h,\eta_h)_{H^1(\Omega_h)}
+(\nabla u_h,\nabla\eta_h)_{L^2(\Omega_h)}-(\partial_n u_h,\eta_h)_{L^2(\partial\Omega_h)}\\
&{}\hspace{5.8em}-(\nabla w_h,\nabla v_h)_{L^2(\Omega_h)}
+(w_h,\partial_n v_h)_{L^2(\partial\Omega_h)}\\
&{}\hspace{5.8em}
+h^2\sum_{T\in\mathcal{T}_h^\Gamma}(\Delta u_h,\Delta v_h)_{L^2(T)}+j_h(u_h,v_h),\\
\ell(\eta_h,v_h):=&{}
(f,\eta_h)_{L^2(\Omega_h)}-h^2\sum_{T\in\mathcal{T}_h^\Gamma}(f, \Delta v_h)_{L^2(T)}.
\end{align*}
We define the norm $\vertiii{\cdot}_h$ on $S_h\times S_h$ by
\[
\vertiii{w_h,u_h}_h^2
:= \|\nabla w_h\|_{L^2(\Omega_h)}^2
 + \|\nabla u_h\|_{L^2(\Omega_h)}^2
 + h^2\sum_{T\in\mathcal{T}_h^\Gamma}
 \|\Delta u_h\|^2_{L^2(T)}
 + j_h(u_h,u_h).
\]
It is direct to deduce that 
\begin{align}
\label{prop1}
\vertiii{u_h,u_h}_h^2\le 2 \vertiii{w_h,u_h}_h^2.
\end{align}
We now establish the continuity and coercivity of the bilinear form $a_h$.

\begin{lemma}[Continuity]
The bilinear form $a_h$ is continuous with respect to the norm $\vertiii{\cdot}_h$, i.e., there exists $C>0$ such that
\begin{equation}\label{2.11}
a_h(w_h,u_h;\eta_h,v_h)
\le C\,\vertiii{w_h,u_h}_h\,\vertiii{\eta_h,v_h}_h.
\end{equation}
\end{lemma}
\begin{proof}
We begin by estimating the terms on $\partial\Omega_h$. 
Note that both $w_h$ and $\eta_h$ vanish on the zero level set $\{\phi_h = 0\}$.
Applying the divergence theorem on the narrow band region $B_h$, defined as the strip bounded by $\partial\Omega_h$ and $\{\phi_h = 0\}$, we obtain
\begin{align*}
(\partial_n u_h, \eta_h)_{L^2(\partial\Omega_h)} 
&=
(\partial_n u_h, \eta_h)_{L^2(\partial B_h)}
= (\nabla u_h, \nabla \eta_h)_{L^2(B_h)}
+ \sum_{T \in \mathcal{T}_h^\Gamma}
(\Delta u_h, \eta_h)_{L^2(B_h \cap T)}
+ \mathcal{E}_1,
\end{align*}
where $\mathcal{E}_1$ collects the edge jump contributions across element edges.
Using the trace and Poincar\'e inequalities \eqref{eq:poincare-inequ}--\eqref{eq:trace-inequ}, the term $\mathcal{E}_1$ can be estimated as
\begin{align*}
|\mathcal{E}_1|
&\le \sum_{E \in \mathcal{F}_h^\Gamma} \big|([\partial_n u_h]_E, \eta_h)_{L^2(E)}\big|
\le C\, j_h(u_h,u_h)^{1/2} \, |\eta_h|_{H^1(\Omega_h)}.
\end{align*}
Thus, we derive
\begin{align*}
|(\partial_n u_h,\eta_h)_{L^2(\partial\Omega_h)}| &\le C \vertiii{w_h,u_h}_h |\eta_h|_{H^1(\Omega_h)},\\
|(w_h,\partial_n v_h)_{L^2(\partial\Omega_h)}|&\le 
C \vertiii{\eta_h,v_h}_h
|w_h|_{H^1(\Omega_h)}.
\end{align*}
The remaining terms in $a_h$ are clearly continuous with respect to the norm $\vertiii{\cdot}_h$. 
This completes the proof of the continuity of $a_h$.
\hfill\end{proof}
\begin{lemma}[Generalized inf-sup condition]\label{lem:inf-sup}
The bilinear form $a_h$ satisfies the following property: 
there exists a constant $\tilde{\beta}>0$ such that, 
for any $(w_h,u_h)\in S_h\times S_h$, 
there exists $(\eta_h,v_h)\in S_h\times S_h$ such that
\begin{equation}
a_h(w_h,u_h;\eta_h,v_h)\ge \tilde{\beta}\vertiii{w_h,u_h}_h\vertiii{\eta_h,v_h}_h.
\end{equation}
\end{lemma}
\begin{proof}
Testing \eqref{2.10} with $(\eta_h,v_h)=(w_h+\alpha u_h,u_h)$, 
where $\alpha>0$ is a constant to be specified later, we obtain
\begin{align}\notag\label{c030}
&a_h(w_h,u_h; w_h+\alpha u_h,u_h)=
(w_h,w_h)_{H^1(\Omega_h)}
+h^2\sum_{T\in\mathcal{T}_h^\Gamma}(\Delta u_h,\Delta u_h)_{L^2(T)}+j_h(u_h,u_h) \\
&\hspace{2.4cm}+\alpha(w_h, u_h)_{H^1(\Omega_h)} + \alpha(\nabla u_h,\nabla u_h)_{L^2(\Omega_h)}-\alpha(\partial_n u_h,u_h)_{L^2(\partial\Omega_h)}.
\end{align}
By requiring $\alpha<1/\sigma$ and substituting the coercivity estimate \eqref{c031} into \eqref{c030}, we obtain
\begin{align}\label{c032}\notag
a_h(w_h&,u_h;w_h+\alpha u_h,u_h)\ge
(w_h,w_h)_{H^1(\Omega_h)}+\alpha(w_h,u_h)_{H^1(\Omega_h)} \\
&+C_{\mathrm{coer}}\alpha 
\Big((\nabla u_h,\nabla u_h)_{L^2(\Omega_h)}
+h^2\sum_{T\in \mathcal{T}_h^\Gamma}(\Delta u_h,\Delta u_h)_{L^2(T)} + j_h(u_h,u_h)\Big).
\end{align}
By invoking the Poincar\'e inequality (Lemma~\ref{poincare2}), we obtain
\begin{equation}\label{c033}
\begin{split}
(w_h,u_h)_{H^1(\Omega_h)}&\ge -C_{p}\|\nabla w_h\|_{L^2(\Omega_h)}\|\nabla u_h\|_{L^2(\Omega_h)}.
\end{split}
\end{equation}
Using Young's inequality to estimate the cross term, we have
\begin{align}\label{c0341}\notag
a_h(w_h,u_h;w_h+\alpha u_h,u_h)
\ge& \Big(C_{\mathrm{coer}}\alpha-\frac{C_{p}^2\alpha^2}{2}\Big)\|\nabla u_h\|^2_{L^2(\Omega_h)}+\frac{1}{2}\|\nabla w_h\|^2_{L^2(\Omega_h)}\\\notag
&+C_{\mathrm{coer}}\alpha\Big(h^2\sum_{T\in\mathcal{T}_h^\Gamma}(\Delta u_h,\Delta u_h)_{L^2(T)}
+j_h(u_h,u_h)\Big)\\
\ge &\beta\vertiii{w_h,u_h}_h^2,
\end{align}
provided that 
$\alpha<\min\big\{2C_{\mathrm{coer}}C_{p}^{-2},{\sigma}^{-1}\big\}$.
By \eqref{prop1}, there exists $\tilde{\beta}>0$, independent of $h$ such that the following estimate holds
\begin{equation}\label{c036}
a_h(w_h,u_h;w_h+\alpha u_h,u_h)\ge\tilde{\beta}\vertiii{w_h,u_h}_h\vertiii{w_h+\alpha u_h,u_h}_h.
\end{equation}   
\hfill\end{proof}

The stability of the scheme \eqref{2.10} then follows from the Babu\v{s}ka--Lax--Milgram theorem as a consequence of the generalized inf--sup condition \eqref{c036} (the non-degeneracy (transpose) condition is automatically satisfied since injectivity implies bijectivity in a finite dimensional space, similar for the Stokes proof later).
\begin{myremark}\upshape
In the original $\phi$-FEM scheme, stability is ensured by choosing the penalty parameter $\sigma$ sufficiently large; see Lemma~\ref{poissonlemma}. 
In contrast, our scheme does not require such a condition. 
Although in the above analysis we introduce a parameter $\alpha>0$ and require it to be sufficiently small, 
this parameter is used solely for the theoretical argument and does not appear in the numerical scheme itself. 
As a result, the proposed scheme is parameter-free.
\end{myremark}

\subsubsection{Consistency and Error Estimate}

Recall the interpolation operator $\tilde I_h$ from Lemma~\ref{interpolation}. 
For $\tilde u$ in Theorem~\ref{thm2.1}, we have $\tilde I_h \tilde u \in S_h$, and it satisfies:
\begin{align}\label{defect1}\notag
&(\nabla\tilde{I}_h\tilde{u},\nabla\eta_h)_{L^2(\Omega_h)}
-(\partial_n\tilde{I}_h\tilde{u},\eta_h)_{L^2(\partial\Omega_h)}
=(\nabla (\tilde{I}_h\tilde{u}-\tilde{u}),\nabla\eta_h)_{L^2(\Omega_h)}\\
&\hspace{5cm}+(-\Delta\tilde{u},\eta_h)_{L^2(\Omega_h)}+(\partial_n (\tilde u-\tilde{I}_h\tilde{u}),\eta_h)_{L^2(\partial\Omega_h)},\\
&(\Delta\tilde{I}_h\tilde{u},\Delta v_h)_{L^2(T)}
=-(\Delta(\tilde{u}-\tilde{I}_h\tilde{u}),
\Delta v_h)_{L^2(T)}
-(-\Delta\tilde{u}, \Delta v_h)_{L^2(T)},\\
\label{defect3}
&j_h(\tilde{I}_h\tilde{u},v_h)= j_h(\tilde{I}_h\tilde{u}-\tilde{u},v_h),
\end{align}
for any $\eta_h,v_h\in S_h$. 
Then, we obtain
\begin{align}\label{def:cons}
a_h(0,\tilde I_h\tilde u; \eta_h,v_h) = r_c(\eta_h,v_h)
+\ell(\eta_h,v_h),
\end{align}
where the consistency error $r_c$ is given by
\begin{align*}
r_c(\eta_h,v_h)=&(\nabla (\tilde{I}_h\tilde{u}-\tilde{u}),\nabla\eta_h)_{L^2(\Omega_h)}
+(\partial_n (\tilde u-\tilde{I}_h\tilde{u}),\eta_h)_{L^2(\partial\Omega_h)}\\
&-(f+\Delta\tilde{u},\eta_h)_{L^2(\Omega_h)}
+j_h(\tilde{I}_h\tilde{u}-\tilde{u},v_h)\\
&+h^2\sum_{T\in\mathcal{T}_h^\Gamma}(f+\Delta\tilde{u},\Delta v_h)_{L^2(T)}
-h^2\sum_{T\in\mathcal{T}_h^\Gamma}(\Delta(\tilde{u}-\tilde{I}_h\tilde{u}),\Delta v_h)_{L^2(T)}.
\end{align*}
Since the auxiliary variable $w_h$ approximates the residual, which vanishes for the exact solution, we consider the error equations satisfied by $(w_h,e_h)$ where $e_h:=u_h-\tilde{I}_h\tilde{u}$. 
Subtracting \eqref{def:cons} from \eqref{2.10}, we obtain
\begin{equation}
a_h(w_h,e_h;\eta_h,v_h)= -r_c(\eta_h,v_h).
\end{equation}

Note that $f+\Delta\tilde{u}$ vanishes in $\Omega$. Therefore, by Lemmas \ref{vanish} and \ref{poincare1}, we can bound the terms involving $f+\Delta\tilde{u}$ in $r_c$ by
\begin{align*}
(f+\Delta\tilde{u},\eta_h)_{L^2(\Omega_h)}&\le\|f+\Delta\tilde{u}\|_{L^2(\Omega_h\backslash\Omega)}\|\eta_h\|_{L^2(\Omega_h\backslash\Omega)}\\
&\le (Ch^{k-1}\|f+\Delta\tilde{u}\|_{H^{k-1}(\Omega_h\backslash\Omega)})(h|\eta_h|_{H^1(\Omega_h^\Gamma)})\\
&\le Ch^k\|f\|_{H^{k-1}(\Omega_h\cup\Omega)}|\eta_h|_{H^1(\Omega_h^\Gamma)},
\end{align*}
where we have used the properties of Stein's extension, and the elliptic regularity to obtain
\begin{align*}
\|\tilde{u}\|_{H^{k+1}(\Omega_h)}\le\|\tilde{u}\|_{H^{k+1}(\Omega_h\cup\Omega)}\le C\|u\|_{H^{k+1}(\Omega)}\le C\|f\|_{H^{k-1}(\Omega)}\le C\|f\|_{H^{k-1}(\Omega_h\cup\Omega)}.
\end{align*}
Similarly, we can derive
\begin{align*}
h^2\sum_{T\in\mathcal{T}_h^\Gamma}(f+\Delta\tilde{u},\Delta v_h)_{L^2(T)}
&\le h^2\|f+\Delta\tilde{u}\|_{L^2(\Omega_h\backslash\Omega)}
\Big(\sum_{T\in\mathcal{T}_h^\Gamma}\|\Delta v_h\|_{L^2(T)}^2\Big)^{\frac12}\\
&\le Ch^2(h^{k-1}\|f+\Delta\tilde{u}\|_{H^{k-1}(\Omega_h\backslash\Omega)})(h^{-1}|v_h|_{H^1(\Omega_h^\Gamma)})\\&\le Ch^k\|f\|_{H^{k-1}(\Omega_h\cup\Omega)}|v_h|_{H^1(\Omega_h^\Gamma)}.
\end{align*}
By the boundary interpolation estimate (3.9a) and the discrete
trace inequality (3.3), we have
\[
\begin{aligned}
&
\left|
\left(
\partial_n(\widetilde u-\widetilde I_h\widetilde u),
\eta_h
\right)_{L^2(\partial\Omega_h)}
\right|                                                       \\
&\quad\leq
\|\nabla(\widetilde u-\widetilde I_h\widetilde u)\|_
 {L^2(\partial\Omega_h)}
\|\eta_h\|_{L^2(\partial\Omega_h)}                            \\
&\quad\leq
C\bigl(h^k+\|\phi-\phi_h\|_{H^1(\Omega_h)}\bigr)
\left(
\|\widetilde u\|_{H^{k+1}(\Omega_h)}
+\|\widetilde u\|_{H^3(\Omega_h)}
\right)
|\eta_h|_{H^1(\Omega_h)}.
\end{aligned}
\]
The remaining terms in $r_c$ can be treated analogously. Therefore, we have
\begin{equation*}
|r_c(\eta_h,v_h)|\le C(h^k+\|\phi-\phi_h\|_{H^1(\Omega_h)})\vertiii{\eta_h,v_h}_h
(\|f\|_{H^{k-1}(\Omega_h\cup \Omega)}+\|f\|_{H^{1}(\Omega_h\cup \Omega)}).
\end{equation*}
From the generalized inf-sup condition in Lemma~\ref{lem:inf-sup}, there exists $(\eta_h,v_h)\in S_h\times S_h$ such that
\begin{align*}
\vertiii{w_h,e_h}_h&\le \frac{1}{\tilde{\beta}}\frac{|a_h(w_h,e_h;\eta_h,v_h)|}
{\vertiii{\eta_h,v_h}_h}
=\frac{1}{\tilde{\beta}}\frac{
|r_c(\eta_h,v_h)|}
{\vertiii{\eta_h,v_h}_h}\\
&\le C(h^k+\|\phi-\phi_h\|_{H^1(\Omega_h)})(\|f\|_{H^{k-1}(\Omega_h\cup \Omega)}+\|f\|_{H^{1}(\Omega_h\cup \Omega)}).
\end{align*}
Combining with the interpolation estimate in Lemma \ref{interpolation}, 
we deduce the main results in Theorem \ref{thm2.1}:
\begin{align}\notag
\|\tilde{u}-u_h&\|_{H^1(\Omega_h)}
\le\|\tilde{u}-\tilde{I}_h\tilde{u}\|_{H^1(\Omega_h)}+\|e_h\|_{H^1(\Omega_h)}\\
&\le C(h^k+\|\phi-\phi_h\|_{H^1(\Omega_h)})(\|f\|_{H^{k-1}(\Omega_h\cup \Omega)}
+ \|f\|_{H^{ 1}(\Omega_h\cup \Omega)}).
\end{align}

\subsection{Proof of Theorem~\ref{thm:stokes}}
In line with the previous sections, we begin by recalling several auxiliary lemmas, in particular the Korn inequality and the inf--sup condition, which will be used in the analysis of the least-squares UnCut FEM for the Stokes problem.
We cite these results with the original numbering from the corresponding references and therefore omit the proofs.
\begin{lemma}[{\!\!\cite[Lemma 3]{phifem_stokes}}]\label{korn}
For all $v_h\in X_h$, it holds that
\begin{equation}\label{eq:korn}
\|\nabla v_h\|_{L^2(\Omega_h)}\le C\|Dv_h\|_{L^2(\Omega_h)}.
\end{equation}
\end{lemma}

\begin{lemma}\label{phifem_stokes}
Assume that Assumptions \ref{a1}-\ref{a3} hold. 
Then, for any $\beta>0$, there exists a constant $0<\gamma<1$ such that for all $v_h\in X_h$ and $q_h\in Q_h$,
\begin{align}\label{c080}\notag
\|Dv_h\|_{L^2(\Omega_h^\Gamma)}^2
+(1-\gamma)h^2|q_h|_{H^1(\Omega_h^\Gamma)}^2
\le{}& \gamma\|Dv_h\|_{L^2(\Omega_h)}^2
+ \beta h^2\sum_{T\in\mathcal{T}_h^\Gamma}
\|\!-\!2\operatorname{div} D v_h+\nabla q_h\|_{L^2(T)}^2 \\
&+\beta\big(\|\nabla\cdot v_h\|_{L^2(\Omega_h^\Gamma)}^2 + j_h^{(2)}(v_h,v_h)\big).
\end{align}
\end{lemma}
\begin{proof}
According to \cite[Lemma~2]{phifem_stokes}, we have
\begin{equation*}
\begin{split}
\|D v_h\|_{L^2(\Omega_h^\Gamma)}^2
&+(1-\gamma)h^2 |q_h|_{H^1(\Omega_h^\Gamma)}^2
\le \gamma \|D v_h\|_{L^2(\Omega_h)}^2 \\
&\quad + \beta\Big(
h^2 \sum_{T\in\mathcal{T}_h^\Gamma}
\|\!-\!\Delta v_h + \nabla q_h\|_{L^2(T)}^2
+ \|\nabla\cdot v_h\|_{L^2(\Omega_h^\Gamma)}^2
+ j_h^{(2)}(v_h,v_h)
\Big).
\end{split}
\end{equation*}
Using the identity
\begin{align*}
2\operatorname{div} D v_h = \Delta v_h + \nabla \operatorname{div} v_h,
\end{align*}
we obtain
\begin{align*}
\sum_{T\in\mathcal{T}_h^\Gamma}
\|\!-\!\Delta v_h + \nabla q_h\|_{L^2(T)}^2\lesssim\sum_{T\in\mathcal{T}_h^\Gamma}
\|\!-\!2\operatorname{div} D v_h + \nabla q_h\|_{L^2(T)}^2
\;
+ \sum_{T\in\mathcal{T}_h^\Gamma}
\|\nabla \operatorname{div} v_h\|_{L^2(T)}^2 .
\end{align*}
The desired result \eqref{c080} follows from the inverse inequality.
\hfill\end{proof}

\begin{lemma}[{\!\!\cite[Lemma 8]{phifem_stokes}}]\label{oldinfsup}
For any $p_h\in \tilde{Q}_h$, where $\tilde{Q}_h$ denotes the space of continuous piecewise $P^{k-1}$ polynomials on $\mathcal{T}_h$ with zero mean over $\Omega$, there exists $v_h^p\in X_h$ such that
\begin{equation}
\begin{split}
\|p_h\|_{L^2(\Omega_h)}^2-Ch^2|p_h|_{H^1(\Omega_h^\Gamma)}^2&\le\int_{\Omega_h}\nabla p_h\cdot v_h^p\, dx,\\
|v_h^p|_{H^1(\Omega_h)}&\le C\|p_h\|_{L^2(\Omega_h)}.
\end{split}
\end{equation}
Here we modify $Q_h$ defined in~\eqref{defQ_H} by changing the constraint to $\int_\Omega p_h\, dx=0$.
The functions are understood to be naturally extended to $\Omega$ if $\Omega\backslash\Omega_h\not=\emptyset$.
\end{lemma}

We note that the inf-sup condition above is stated for piecewise polynomial functions $p_h$ satisfying the constraint $\int_{\Omega} p_h\,dx = 0$, which can be somewhat inconvenient to enforce in practical computations.
To circumvent this difficulty, we replace this requirement with a zero-mean condition over $\Omega_h$, namely $\int_{\Omega_h} p_h\,dx = 0$, which is more convenient to impose numerically. 
The corresponding inf--sup condition is stated as follows.
\begin{lemma}[New inf-sup pair]\label{newinfsup}
For all $p_h\in Q_h$, where $Q_h$ is defined in \eqref{defQ_H}, there exists $v_h^p\in X_h$ such that
\begin{equation}\label{66}
\begin{split}
\|p_h\|_{L^2(\Omega_h)}^2-Ch^2|p_h|_{H^1(\Omega_h^\Gamma)}^2&\le\int_{\Omega_h}\nabla p_h\cdot v_h^p\, dx,\\
|v_h^p|_{H^1(\Omega_h)}&\le C\|p_h\|_{L^2(\Omega_h)}.
\end{split}
\end{equation}        
\end{lemma}
\begin{proof}
The proof is based on the result of Lemma~\ref{oldinfsup}. 
Let $p_h\in Q_h$, which satisfies $\int_{\Omega_h}p_h\, dx=0$. Since $p_h$ is only defined on $\Omega_h$ instead of $\Omega_h\cup\Omega$, we use a standard discrete finite element extension, and continue
to denote the extended function by $p_h$. The extension agrees with
the original function on $\Omega_h$ and satisfies
\[
\|p_h\|_{L^2(\Omega_h\cup\Omega)}
\leq
C_{\rm ext}\|p_h\|_{L^2(\Omega_h)},
\tag{3.30a}
\]
where $C_{\rm ext}$ is independent of $h$ and of the cut
configuration; see \cite[Sections 2.1--2.2, Lemma 2.1]{BHL2022NM}. Consider the following modified function: 
\begin{equation}\label{decomp}
\hat{p}_h:=p_h+c_p   
\end{equation}
where $c_p$ is a constant given by
\begin{equation}\label{def:cp}
c_p:=\frac{1}{|\Omega|}
\Big(\int_{\Omega_h}p_h\, dx-
\int_{\Omega}p_h\, dx\Big).         
\end{equation}
Since $\int_{\Omega_h} p_h\, dx = 0$, $p_h$ is orthogonal to constants in $L^2(\Omega_h)$. Thus, \eqref{decomp} gives an orthogonal decomposition in $L^2(\Omega_h)$. Then we have
\begin{equation}\label{72}
\|p_h+c_p\|_{L^2(\Omega_h)}^2
=\|p_h\|_{L^2(\Omega_h)}^2+c_p^2|\Omega_h|\ge\|p_h\|_{L^2(\Omega_h)}^2.
\end{equation}
Further, $\hat{p}_h$ is still a continuous piecewise polynomial function, and
\begin{equation}
\int_{\Omega}\hat{p}_h\, dx=\int_{\Omega}p_h\, dx
+ |\Omega| c_p=\int_{\Omega_h}p_h\, dx=0.
\end{equation}
By Lemma \ref{oldinfsup}, there exists $\hat{v}_h^p\in X_h$ such that
\begin{equation}\label{68}
\begin{split}    \|\hat{p}_h\|_{L^2(\Omega_h)}^2-Ch^2|\hat{p}_h|_{H^1(\Omega_h^\Gamma)}^2&\le\int_{\Omega_h}\nabla\hat{p}_h\cdot\hat{v}_h^p\, dx,\\
\|\nabla\hat{v}_h^p\|_{L^2(\Omega_h)}&\le C\|\hat{p}_h\|_{L^2(\Omega_h)}.
\end{split}
\end{equation}
We claim that $(\hat{v}_h^p,p_h)$ is exactly the inf-sup pair that satisfies (\ref{66}). Notice that $\hat{p}_h$ and $p_h$ differ by a constant, so we have $\nabla\hat{p}_h\equiv\nabla p_h$. From the first equation in (\ref{68}) we get
\begin{equation}\label{71}
\begin{split}
\|p_h+c_p\|^2_{L^2(\Omega_h)}&\le \int_{\Omega_h}\nabla(p_h+c_p)\cdot\hat{v}_h^p \, dx +Ch^2|p_h+c_p|_{H^1(\Omega_h^\Gamma)}^2\\
&=\int_{\Omega_h}\nabla p_h\cdot\hat{v}_h^p\, dx
+Ch^2|p_h|_{H^1(\Omega_h^\Gamma)}^2.
\end{split}                
\end{equation}
Substituting the bound (\ref{72}) into (\ref{71}) we obtain the first inequality in (\ref{66}). For the second inequality in (\ref{66}), we have
\begin{subequations}\label{eq:vp-cp-bounds}

\begin{equation}
\begin{aligned}
\|\nabla\widehat v_h^p\|_{L^2(\Omega_h)}^2
&\le C\|\widehat p_h\|_{L^2(\Omega_h)}^2= C\left(
    \|p_h\|_{L^2(\Omega_h)}^2
    +c_p^2|\Omega_h|
    \right).
\end{aligned}
\label{eq:vp-bound}
\end{equation}

By the definition of \(c_p\) in \eqref{def:cp}, together with the
stability of the discrete extension, we obtain
\begin{equation}
\begin{aligned}
|c_p|
&=
\frac{1}{|\Omega|}
\left|
\int_{\Omega_h}p_h\,dx-\int_{\Omega}p_h\,dx
\right| \\
&\le
\frac{
\bigl(
|\Omega_h\setminus\Omega|
+|\Omega\setminus\Omega_h|
\bigr)^{1/2}
}{
|\Omega|
}
\|p_h\|_{L^2(\Omega_h\cup\Omega)} \\
&\le
\frac{
C\bigl(
|\Omega_h\setminus\Omega|
+|\Omega\setminus\Omega_h|
\bigr)^{1/2}
}{
|\Omega|
}
\|p_h\|_{L^2(\Omega_h)} .
\end{aligned}
\label{eq:cp-bound}
\end{equation}

\end{subequations}

Substituting \eqref{eq:cp-bound} into \eqref{eq:vp-bound}
yields the second inequality in \eqref{66}. Consequently,
\((\widehat v_h^p,p_h)\) is the desired inf--sup pair.

\hfill\end{proof}
\subsubsection{Stability}
We rewrite scheme~\eqref{3.7} into the following abstract variational form.
Find $(w_h,u_h,p_h)\in X_h\times X_h\times Q_h$ such that
\begin{equation}\label{c077}
A_h(w_h,u_h,p_h;\eta_h,v_h,q_h) = L(\eta_h,v_h,q_h)
\end{equation}
for all $(\eta_h,v_h,q_h)\in X_h\times X_h\times Q_h$,
where
\begin{align}
\label{Stoke:ah}\notag
A_h(w_h,u_h,p_h;{}&{}\eta_h,v_h,q_h)
=(w_h,\eta_h)_{L^2(\Omega_h)}
+(\nabla w_h,\nabla \eta_h)_{L^2(\Omega_h)}
+(\nabla\!\cdot u_h,\nabla\!\cdot v_h)_{L^2(\Omega_h)}\\\notag
&+2(Du_h,D\eta_h)_{L^2(\Omega_h)}
-(p_h,\nabla\!\cdot\eta_h)_{L^2(\Omega_h)}
-\bigl((2Du_h-p_h I)n,\eta_h\bigr)_{L^2(\partial\Omega_h)} \\\notag
&-2(Dw_h,Dv_h)_{L^2(\Omega_h)}
+(\nabla\!\cdot w_h,q_h)_{L^2(\Omega_h)}
+\bigl(w_h,(2Dv_h-q_h I)n\bigr)_{L^2(\partial\Omega_h)} \\
&+h^2\sum_{T\in\mathcal{T}_h^\Gamma}(-2 \textrm{div}D u_h+\nabla p_h,
-2\textrm{div}D v_h+\nabla q_h)_{L^2(T)}
+j_h^{(2)}(u_h,v_h),\\
L(\eta_h,v_h,q_h){}&=(f,\eta_h)_{L^2(\Omega_h)}
+h^2\sum_{T\in\mathcal{T}_h^\Gamma}(f,-2\textrm{div}D v_h+\nabla q_h)_{L^2(T)}.
\end{align}
Similar to the Poisson case, the stability follows from continuity and the generalized inf-sup condition of the bilinear form $A_h$.
We first introduce the norm $\vertiii{\cdot}_h$ on $X_h\times X_h\times Q_h$ by
\begin{align*}
\vertiii{w_h,u_h,p_h}_h^2
&=\|\nabla w_h\|_{L^2(\Omega_h)}^2
 +\|D u_h\|_{L^2(\Omega_h)}^2
 +\|\nabla\!\cdot u_h\|_{L^2(\Omega_h)}^2\\
&+\|p_h\|_{L^2(\Omega_h)}^2
 +h^2\sum_{T\in\mathcal{T}_h^\Gamma}\|-2\operatorname{div} D u_h+\nabla p_h\|_{L^2(T)}^2+j_h^{(2)}(u_h,u_h).
\end{align*}
Then the continuity and a generalized inf--sup condition for $A_h$ can be established.
\begin{lemma}[Continuity]
The bilinear form $A_h$ in \eqref{Stoke:ah} is continuous with respect to the norm $\vertiii{\cdot}_h$, i.e.,
there exists $C>0$ such that
\begin{equation}
A_h(w_h,u_h,p_h;\eta_h,v_h,q_h)
\le C\,\vertiii{w_h,u_h,p_h}_h\,
        \vertiii{\eta_h,v_h,q_h}_h
\end{equation}
for all $(w_h,u_h,p_h),(\eta_h,v_h,q_h)\in X_h\times X_h\times Q_h$.
\end{lemma}
    \begin{proof}
We focus on proving the continuity of the boundary term in the bilinear form \eqref{Stoke:ah}, since the continuity of the remaining terms in $A_h$ is straightforward.
By the trace and Poincar\'e inequalities, we obtain
\begin{align}\label{bnd:trace}\notag
-((2Du_h-p_hI)n,\eta_h)_{L^2(\partial\Omega_h)}
&\le \sqrt{h}\,\| (2Du_h-p_hI)n\|_{L^2(\partial\Omega_h)}\;h^{-\frac12}{\|\eta_h\|_{L^2(\partial\Omega_h)}}\\\notag
&\le C \|2Du_h-p_hI\|_{L^2(\Omega_h^\Gamma)}\,|\eta_h|_{H^1(\Omega_h^\Gamma)}\\
&\le C\big(\|Du_h\|_{L^2(\Omega_h)}+\|p_h\|_{L^2(\Omega_h)}\big)\,|\eta_h|_{H^1(\Omega_h)}.
\end{align}
An analogous estimate holds for $(w_h,(2Dv_h-q_hI)n)_{L^2(\partial\Omega_h)}$, and the result follows.
\hfill\end{proof}
\begin{lemma}[Generalized inf-sup condition]\label{Stokes:inf-sup}
The bilinear form~\eqref{Stoke:ah} satisfies a generalized inf--sup condition: there exists a constant $c_{\mathrm{infsup}}>0$ such that, for every $(w_h,u_h,p_h)\in X_h\times X_h\times Q_h$, there exists $(\eta_h,v_h,q_h)\in X_h\times X_h\times Q_h$ satisfying
\begin{equation}
A_h(w_h,u_h,p_h;\eta_h,v_h,q_h)\ge c_{\mathrm{infsup}}\vertiii{w_h,u_h,p_h}_h\vertiii{\eta_h,v_h,q_h}_h.
\end{equation}
\end{lemma}
\begin{proof}
Testing~\eqref{c077} with $(\eta_h,v_h,q_h)=(w_h+\alpha u_h,u_h,p_h)$ leads to
\begin{align}
A_h&(w_h,u_h,p_h;w_h+\alpha u_h,u_h,p_h) \nonumber\\
={}&(w_h,w_h)_{H^1(\Omega_h)}
+\alpha(w_h,u_h)_{H^1(\Omega_h)}
+2\alpha\|D u_h\|_{L^2(\Omega_h)}^2 \nonumber\\
&-\alpha(p_h,\nabla\!\cdot u_h)_{L^2(\Omega_h)}
-\alpha\bigl((2D u_h-p_h I)n,u_h\bigr)_{L^2(\partial\Omega_h)} \nonumber\\
&+\|\nabla\!\cdot u_h\|_{L^2(\Omega_h)}^2
+h^2
\sum_{T\in \mathcal{T}_h^\Gamma}\|2\operatorname{div} D u_h-\nabla p_h\|_{L^2(T)}^2
+j_h^{(2)}(u_h,u_h).
\label{3.21}
\end{align}
First, by Korn's inequality~\eqref{eq:korn} and Young's inequality, we obtain
\begin{align*}
    \alpha(w_h,u_h)_{H^1(\Omega_h)}-\alpha (p_h,\nabla\!\cdot u_h)_{L^2(\Omega_h)}
    &\ge -\frac12 \|\nabla w_h\|_{L^2(\Omega_h)}^2 
    - \frac{1}{8} \|\nabla\cdot u_h\|_{L^2(\Omega_h)}^2 \\
&    - C\alpha^2 (\|Du_h\|_{L^2(\Omega_h)}^2+\|p_h\|_{L^2(\Omega_h)}^2).
\end{align*}
We now estimate the boundary term $((2D u_h-p_h I)n,u_h)_{L^2(\partial\Omega_h)}$. 
The treatment of this term exploits the property that $u_h = 0$ on $\{\phi_h = 0\}$ and applies the divergence theorem to convert the boundary integral over $\partial\Omega_h$ into a domain integral over $B_h$, i.e., 
\begin{align*}
((2Du_h\!-\!p_hI)n&,u_h)_{L^2(\partial\Omega_h)} 
= ((2Du_h\!-\!p_hI)n,u_h)_{L^2(\partial B_h)}\\
&=\sum_{T\in \mathcal{T}_h^\Gamma}((2\text{div}Du_h
\!-\!\nabla p_h),u_h)_{L^2(T\cap B_h)}
+ (2Du_h\!-\!p_hI,\nabla u_h)_{L^2(B_h)}
+\mathcal{E},
\end{align*}
where $\mathcal{E}$ collects the edge jumps of $(Du_h)n$ across element edges and, according to Lemma~\ref{poincare1} and \eqref{defj2}, it satisfies
\begin{align*}
|\mathcal{E}|
\le \sum_{E\in\mathcal{F}_h^\Gamma}
|([2(Du_h)n]_E,u_h)_{L^2(E)}|
\le C(j_h^{(2)}(u_h,u_h))^{\frac12}
\| \nabla u_h\|_{L^2(\Omega_h^\Gamma)}.
\end{align*}
Thus, by using the Poincar\'e inequality~\eqref{eq:poincare-inequ}, the trace inequality~\eqref{eq:trace-inequ}, and Korn's inequality~\eqref{eq:korn}, we obtain
\begin{align}
&-\alpha|((2Du_h-p_hI)n,u_h)_{L^2(\partial\Omega_h)}| \notag\\
&\ge{}-\frac{h^2}{2}
\sum_{T\in \mathcal{T}_h^\Gamma}
\left\|2\text{div}Du_h-\nabla p_h\right\|_{L^2(T)}^2
-C\alpha^2\|Du_h\|_{L^2(\Omega_h)}^2\notag\\
&\hspace{1.5cm} 
-\frac{1}{2}j_h^{(2)}(u_h,u_h)
-\frac{1}{2}\|\nabla\cdot u_h\|_{L^2(\Omega_h^\Gamma)}^2
-{C}\alpha^2\|p_h\|_{L^2(\Omega_h^\Gamma)}^2
-2\alpha\|Du_h\|_{L^2(\Omega_h^\Gamma)}^2.\label{eq:344}
\end{align}
We emphasize that it is crucial to control the coefficient of $\|D u_h\|_{L^2(\Omega_h)}^2$ by $C\alpha^2$, 
which allows it to be absorbed into the positive term $2\alpha\|D u_h\|_{L^2(\Omega_h)}^2$ in \eqref{3.21}. 
For the treatment of $2\alpha\|Du_h\|_{L^2(\Omega_h^\Gamma)}^2$ in \eqref{eq:344}, we invoke the following result:
by Lemma \ref{phifem_stokes} with $\beta=\frac{1}{4}$, there exists $0<\gamma<1$ such that
\begin{align}
-\|D u_h\|_{L^2(\Omega_h^\Gamma)}^2
\ge{}\;
&-\gamma\|D u_h\|_{L^2(\Omega_h)}^2 
-\frac{h^2}{4}\sum_{T\in\mathcal{T}_h^\Gamma}\|2\operatorname{div} D u_h-\nabla p_h\|_{L^2(T)}^2 \nonumber\\
&-\frac{1}{4}\|\nabla\!\cdot u_h\|_{L^2(\Omega_h^\Gamma)}^2
-\frac{1}{4}j_h^{(2)}(u_h,u_h)
+(1-\gamma)h^2|p_h|_{H^1(\Omega_h^\Gamma)}^2.
\label{3.22}
\end{align}
Combining \eqref{3.21}--\eqref{3.22}, and let $\alpha<\frac{1}{4}$, we derive
\begin{align}
A_h&(w_h,u_h,p_h;w_h+\alpha u_h,u_h,p_h) 
\ge{}
\frac12\|\nabla w_h\|_{L^2(\Omega_h)}^2
-\hat{C}\alpha^2\|p_h\|_{L^2(\Omega_h)}^2 \nonumber\\
&+\bigl(2(1-\gamma)\alpha-C\alpha^2\bigr)\|D u_h\|_{L^2(\Omega_h)}^2
+2(1-\gamma)\alpha h^2|p_h|_{H^1(\Omega_h^\Gamma)}^2 \nonumber\\
&+\frac14|\nabla\!\cdot u_h\|_{L^2(\Omega_h)}^2 
+\frac{h^2}{4}\sum_{T\in\mathcal{T}_h^\Gamma}\|2\operatorname{div} D u_h-\nabla p_h\|_{L^2(T)}^2
+\frac14 j_h^{(2)}(u_h,u_h).
\label{3.23}
\end{align}
To control the negative term $-\hat C \alpha^2 \|p_h\|_{L^2(\Omega_h)}^2$ in~\eqref{3.23}, we invoke Lemma~\ref{newinfsup} and use the associated inf--sup pair $(v_h^p,p_h)$.
By choosing the test functions $(\eta_h,v_h,q_h)=(v_h^p,0,0)$ and estimating the boundary term using the trace inequality as in \eqref{bnd:trace}, it is straightforward to deduce
\begin{align}\label{3.25}\notag
&A_h(w_h,u_h,p_h;v_h^p,0,0)\\\notag
={}&(w_h,v_h^p)_{H^1(\Omega_h)}+2(Du_h,Dv_h^p)_{L^2(\Omega_h)}
-2(Du_h n,v_h^p)_{L^2(\partial\Omega_h)}+(\nabla p_h,v_h^p)_{L^2(\Omega_h)}\\
\ge{}&\frac{1}{2}\|p_h\|_{L^2(\Omega_h)}^2-C\big(h^2|p_h|_{H^1(\Omega_h^\Gamma)}^2+\|\nabla w_h\|_{L^2(\Omega_h)}^2+\|Du_h\|_{L^2(\Omega_h)}^2\big),  
\end{align}
where we used \eqref{66} and Young's inequality in the last step.
Adding $3\hat{C}\alpha^2$ times~\eqref{3.25} to~\eqref{3.23}, and choosing $\alpha>0$ sufficiently small, we obtain
\begin{align}
A_h&(w_h,u_h,p_h;
     w_h+\alpha u_h+3\hat{C}\alpha^2 v_h^p,\,
     u_h,p_h) \nonumber\\
\ge{}\;&(\frac12-3C\hat{C}\alpha^2)\|\nabla w_h\|_{L^2(\Omega_h)}^2
+\frac{\hat{C}}{2}\alpha^2\|p_h\|_{L^2(\Omega_h)}^2 \nonumber\\
&+\bigl(2(1-\gamma)\alpha-(C+3C\hat{C})\alpha^2\bigr)\|D u_h\|_{L^2(\Omega_h)}^2 
+\bigl((1-\gamma)2\alpha-3C\hat{C}\alpha^2\bigr)
  h^2|p_h|_{H^1(\Omega_h^\Gamma)}^2 \nonumber\\
&+\frac14\|\nabla\!\cdot u_h\|_{L^2(\Omega_h)}^2
+\frac{h^2}{4}\sum_{T\in\mathcal{T}_h^\Gamma}\|2\operatorname{div} D u_h-\nabla p_h\|_{L^2(T)}^2
+\frac14 j_h^{(2)}(u_h,u_h) \nonumber\\
\ge{}\;&c_{\mathrm{infsup}}\,
\vertiii{w_h,u_h,p_h}_h\,
\vertiii{w_h+\alpha u_h+3\hat{C}\alpha^2 v_h^p,\,u_h,p_h}_h,
\label{c083}
\end{align}
where we used Lemma \ref{korn} in the last step. Since the trial and test spaces coincide and are finite-dimensional,
the generalized inf--sup condition implies injectivity and hence
bijectivity. Therefore, the transpose non-degeneracy condition and
the unique solvability of the discrete Stokes problem follow.
\hfill\end{proof}
    
\subsubsection{Consistency and Error Estimate}
We begin by deriving the error equations. 
Let $\tilde u\in H^{k+1}(\Omega_h\cup\Omega)\cap H^{3}(\Omega_h\cup\Omega)$ and $\tilde p\in H^{k}(\Omega_h\cup\Omega)$ be the Stein extensions of the solution $(u,p)$ to $\Omega_h\cup\Omega$, respectively.
We obtain
\begin{align}\label{stokes:regularity}
\|\tilde u\|_{H^{k+1}(\Omega_h)}
+\|\tilde u\|_{H^{3}(\Omega_h)}
+\|\tilde p\|_{H^{k}(\Omega_h)}
\le C(\|f\|_{H^{k-1}(\Omega_h\cup\Omega)}+
\|f\|_{H^{1}(\Omega_h\cup\Omega)}).
\end{align}
Let $\tilde I_h$ be the operator introduced in Lemma~\ref{interpolation}
and $I_h$ the standard interpolation on $\Omega_h$. 
However, in general $I_h\tilde p \notin Q_h$, since its mean over $\Omega_h$ need not vanish. 
Thus, we use $\Pi$ in \eqref{def:Pi} and consider $\Pi I_h\tilde p\in Q_h$.
Then one can define the consistency by substituting the interpolants $\tilde I_h\tilde u$ and projection $\Pi I_h\tilde p$ into the scheme, which leads to
\begin{align} 
A_h(0,\tilde I_h \tilde u, \Pi I_h \tilde p;\eta_h,v_h,q_h)=R_c(\eta_h,v_h,q_h)+L(\eta_h,v_h,q_h),
\label{c085}
\end{align}
where $R_c(\eta_h,v_h,q_h) =
R_1(\eta_h,v_h,q_h) +R_2(\eta_h,v_h,q_h)$ with 
\begin{align}\label{def:R1}
R_1(\eta_h,v_h,q_h)
{}&=2\bigl(D(\tilde{I}_h\tilde{u}-\tilde{u}),D\eta_h\bigr)_{L^2(\Omega_h)}
-\bigl(\Pi I_h \tilde p -\tilde{p},\nabla\!\cdot\eta_h\bigr)_{L^2(\Omega_h)} \nonumber\\
&-\bigl((2D(\tilde{I}_h\tilde{u}-\tilde{u})
-(\Pi I_h \tilde p-\tilde{p})I)n,\eta_h\bigr)_{L^2(\partial\Omega_h)} 
+ \bigl(\nabla\!\cdot(\tilde{I}_h\tilde{u} - \tilde u),\nabla\!\cdot v_h\bigr)_{L^2(\Omega_h)} \nonumber\\
&+h^2\sum_{T\in\mathcal{T}_h^\Gamma}\bigl(-2\operatorname{div} D(\tilde{I}_h\tilde{u}-\tilde{u})
+\nabla(\Pi I_h \tilde p-\tilde{p}),\,
-2\operatorname{div} D v_h+\nabla q_h
\bigr)_{L^2(T)} \nonumber\\
&+j_h^{(2)}(\tilde{I}_h\tilde{u}-\tilde{u},v_h),
\end{align}
and the consistency error due to the extension writes
\begin{align}\notag
R_2(\eta_h,v_h,q_h)
&:=(\!-\!f+(-2\text{div}
D\tilde{u}+\nabla\tilde{p}),\eta_h)_{L^2(\Omega_h)}
+(\nabla\cdot\tilde{u},\nabla\!\cdot v_h)_{L^2(\Omega_h)}\\
&\quad+h^2\sum_{T\in\mathcal{T}_h^\Gamma}
\big(
\!-\!f+(-2\text{div}
D\tilde{u}+\nabla\tilde{p}),\, -2\operatorname{div} D v_h+\nabla q_h
\big)_{L^2(T)}.
\label{c088}
\end{align}
To estimate $R_1$ in \eqref{def:R1}, we bound the terms involving $\tilde u$ by the interpolation estimate in Lemma~\ref{interpolation}. These estimates are standard. In particular, applying (3.9a) componentwise and using (3.3), we have \[ \begin{aligned} \left| \left( 2D(\widetilde I_h\widetilde u-\widetilde u)n,\eta_h \right)_{L^2(\partial\Omega_h)} \right| &\leq C\bigl(h^k+\|\phi-\phi_h\|_{H^1(\Omega_h)}\bigr) \\ &\quad\times \left( \|\widetilde u\|_{H^{k+1}(\Omega_h)} +\|\widetilde u\|_{H^3(\Omega_h)} \right) |\eta_h|_{H^1(\Omega_h)}. \end{aligned}\]
We next estimate the terms involving the pressure. Since the
projection $\Pi$ changes a function only by a constant, integration
by parts gives
\[
\begin{aligned}
&-
(\Pi I_h\widetilde p-\widetilde p,\nabla\cdot\eta_h)_
 {L^2(\Omega_h)}
+
((\Pi I_h\widetilde p-\widetilde p)n,\eta_h)_
 {L^2(\partial\Omega_h)}
\\
&\qquad
=
(\nabla(I_h\widetilde p-\widetilde p),\eta_h)_
 {L^2(\Omega_h)}
\\
&\qquad
=
-
(I_h\widetilde p-\widetilde p,\nabla\cdot\eta_h)_
 {L^2(\Omega_h)}
+
((I_h\widetilde p-\widetilde p)n,\eta_h)_
 {L^2(\partial\Omega_h)}.
\end{aligned}
\]
Consequently, by the scaled trace inequality (3.3a), the discrete
trace inequality (3.3), and the standard interpolation estimates,
\[
\begin{aligned}
&
\left|
-
(I_h\widetilde p-\widetilde p,\nabla\cdot\eta_h)_
 {L^2(\Omega_h)}
+
((I_h\widetilde p-\widetilde p)n,\eta_h)_
 {L^2(\partial\Omega_h)}
\right|
\\
&\quad\leq
\|I_h\widetilde p-\widetilde p\|_{L^2(\Omega_h)}
|\eta_h|_{H^1(\Omega_h)}
\\
&\qquad
+
\|I_h\widetilde p-\widetilde p\|_{L^2(\partial\Omega_h)}
\|\eta_h\|_{L^2(\partial\Omega_h)}
\\
&\quad\leq
C\Bigl(
\|I_h\widetilde p-\widetilde p\|_{L^2(\Omega_h)}
+
\|I_h\widetilde p-\widetilde p\|_{L^2(\Omega_h^\Gamma)}
\\
&\hspace{38mm}
+
h\|\nabla(I_h\widetilde p-\widetilde p)\|_
 {L^2(\Omega_h^\Gamma)}
\Bigr)
|\eta_h|_{H^1(\Omega_h)}
\\
&\quad\leq
Ch^k\|\widetilde p\|_{H^k(\Omega_h)}
|\eta_h|_{H^1(\Omega_h)}.
\end{aligned}
\]
Combining with \eqref{stokes:regularity}, we have
\begin{align}\notag
&|R_1(\eta_h,v_h,q_h)|
\le C(h^k+\|\phi-\phi_h\|_{H^1(\Omega_h)})
(\|f\|_{H^{k-1}(\Omega_h\cup\Omega)}
+\|f\|_{H^{1}(\Omega_h\cup\Omega)})
\vertiii{\eta_h,v_h,q_h}_h.
\end{align}
Moreover, to estimate $R_2$, we invoke Lemma~\ref{vanish}. It is then straightforward to estimate
\begin{align*}
\|-f+(-2\, \textrm{div}
D\tilde{u}+\nabla\tilde{p})\|_{L^2(\Omega_h)}
&=\|-f+(-2\,\textrm{div} D\tilde{u}+\nabla\tilde{p})\|_{L^2(\Omega_h\setminus\Omega)}\\
&\le Ch^{k-1}\|-f+(-2\,\textrm{div} D\tilde{u}+\nabla\tilde{p})\|_{H^{k-1}(\Omega_h\setminus\Omega)}\\
&\le Ch^{k-1}\big(\|f\|_{H^{k-1}(\Omega_h)}+\|\tilde{u}\|_{H^{k+1}(\Omega_h)}+\|\tilde{p}\|_{H^{k}(\Omega_h)}\big),\\
\|\nabla\!\cdot\tilde u\|_{L^2(\Omega_h^\Gamma)}
=\|\nabla\!\cdot\tilde u\|_{L^2(\Omega_h\setminus\Omega)}
&\le C h^{k}\,\|\nabla\!\cdot\tilde u\|_{H^{k}(\Omega_h\setminus\Omega)}
\le C h^{k}\,\|\tilde u\|_{H^{k+1}(\Omega_h)}.
\end{align*}
Collecting the above estimates, we obtain the bound
\begin{align}\label{eq:354}\notag
|R_c(\eta_h,v_h,q_h)|
\le C(h^k+\|\phi-\phi_h\|_{H^1(\Omega_h)})
(\|f\|_{H^{k-1}(\Omega_h\cup\Omega)}
+\|f\|_{H^{1}(\Omega_h\cup\Omega)})
\vertiii{\eta_h,v_h,q_h}_h.
\end{align}
To derive the error equation, we subtract \eqref{c085} from \eqref{c077}. This yields the error equation for
$e_u:=u_h-\tilde{I}_h\tilde{u}$ and $e_p:=p_h-\Pi I_h\tilde{p}$, which can be written as
\begin{equation}
A_h(w_h,e_u,e_p;\eta_h,v_h,q_h)=-R_c(\eta_h,v_h,q_h).
\end{equation} 
Using the generalized inf-sup condition in Lemma~\ref{Stokes:inf-sup}, there exists $(\eta_h,{v}_h,q_h)\in X_h\times X_h\times Q_h$ such that
\begin{align}\notag
\vertiii{w_h,e_u,e_p}_h
&\le \frac{1}{c_{\mathrm{infsup}}}\frac{|R_c(\eta_h,{v}_h,q_h)|}{\vertiii{\eta_h,{v}_h,q_h}_h}\\
&\le C(h^k+\|\phi-\phi_h\|_{H^1(\Omega_h)})
(\|f\|_{H^{k-1}(\Omega_h\cup\Omega)}
+\|f\|_{H^{1}(\Omega_h\cup\Omega)}).
\end{align}
Combining the estimate above with the interpolation estimates in
Lemma~\ref{interpolation}, and using the triangle inequality, we obtain
\begin{align*}
\|\widetilde u-u_h\|_{H^1(\Omega_h)}
&\le
\|\widetilde u-\widetilde I_h\widetilde u\|_{H^1(\Omega_h)}
+\|e_u\|_{H^1(\Omega_h)}
\\
&\le
C\bigl(h^k+\|\phi-\phi_h\|_{H^1(\Omega_h)}\bigr)
\Bigl(
\|f\|_{H^{k-1}(\Omega_h\cup\Omega)}
+\|f\|_{H^1(\Omega_h\cup\Omega)}
\Bigr),
\\[0.3em]
\|\Pi\widetilde p-p_h\|_{L^2(\Omega_h)}
&\le
\|\Pi(\widetilde p-I_h\widetilde p)\|_{L^2(\Omega_h)}
+\|\Pi I_h\widetilde p-p_h\|_{L^2(\Omega_h)}
\\
&\le
C\|\widetilde p-I_h\widetilde p\|_{L^2(\Omega_h)}
+\|e_p\|_{L^2(\Omega_h)}
\\
&\le
C\bigl(h^k+\|\phi-\phi_h\|_{H^1(\Omega_h)}\bigr)
\Bigl(
\|f\|_{H^{k-1}(\Omega_h\cup\Omega)}
+\|f\|_{H^1(\Omega_h\cup\Omega)}
\Bigr).
\end{align*}
Here, in the pressure estimate, we used the \(L^2\)-stability of
\(\Pi\), together with
\(e_p=p_h-\Pi I_h\widetilde p\).
    
\section{Numerical Experiments}
In this section, we present numerical experiments to illustrate the performance of the proposed
least-squares UnCut FEM and to support the theoretical results established in the previous sections.
In particular, we demonstrate the convergence behavior of the numerical approximations and
the robustness of the method with respect to the geometric approximation of the interface.
All numerical experiments are implemented using the open-source finite element software
\texttt{FEniCS}.

\subsection{Poisson Equation in Two Dimensions}

In this example, the computational domain is defined by
\[
\Omega = \{(x,y)\in\mathbb{R}^2 : (x-0.5)^2 + (y-0.5)^2 < 0.125\}.
\]
A sequence of uniformly refined triangular meshes with mesh size $h$ is generated on the background domain $[0,1]\times[0,1]$.
We begin with a uniform triangulation (with initial mesh size $h=0.1\sqrt{2}$) and then apply successive uniform refinements to obtain the mesh sequence used in the convergence study.
The exact solution $u_{\mathrm{exact}}$ is defined on $[0,1]\times[0,1]$ as
\[
u_{\mathrm{exact}} =
\bigl(0.125 - (x-0.5)^2 - (y-0.5)^2\bigr)\,
\exp(x)\,\sin(2\pi y).
\]
The corresponding source term is given by
\[
f_{\mathrm{exact}} = -\Delta u_{\mathrm{exact}}.
\]
The error between the numerical solution and the exact solution $u_{\mathrm{exact}}$ is measured on $\Omega_h$.

We implement the scheme \eqref{2.7} using polynomial degrees $k=1$ ($P^1$ elements) and $k=2$ ($P^2$ elements).  
The resulting convergence orders are shown in Figures~\ref{fig1} and~\ref{fig2}, respectively.  
As demonstrated by the numerical results, the proposed scheme exhibits not only the optimal $H^1$ convergence rate---consistent with the theoretical analysis---but also an optimal $L^2$ convergence rate in practice (observed but not proved, similar for the other numerical experiments).
\begin{figure}[!htbp]
    \centering
    \begin{subfigure}{0.48\textwidth}
        \centering
        \includegraphics[width=\linewidth]{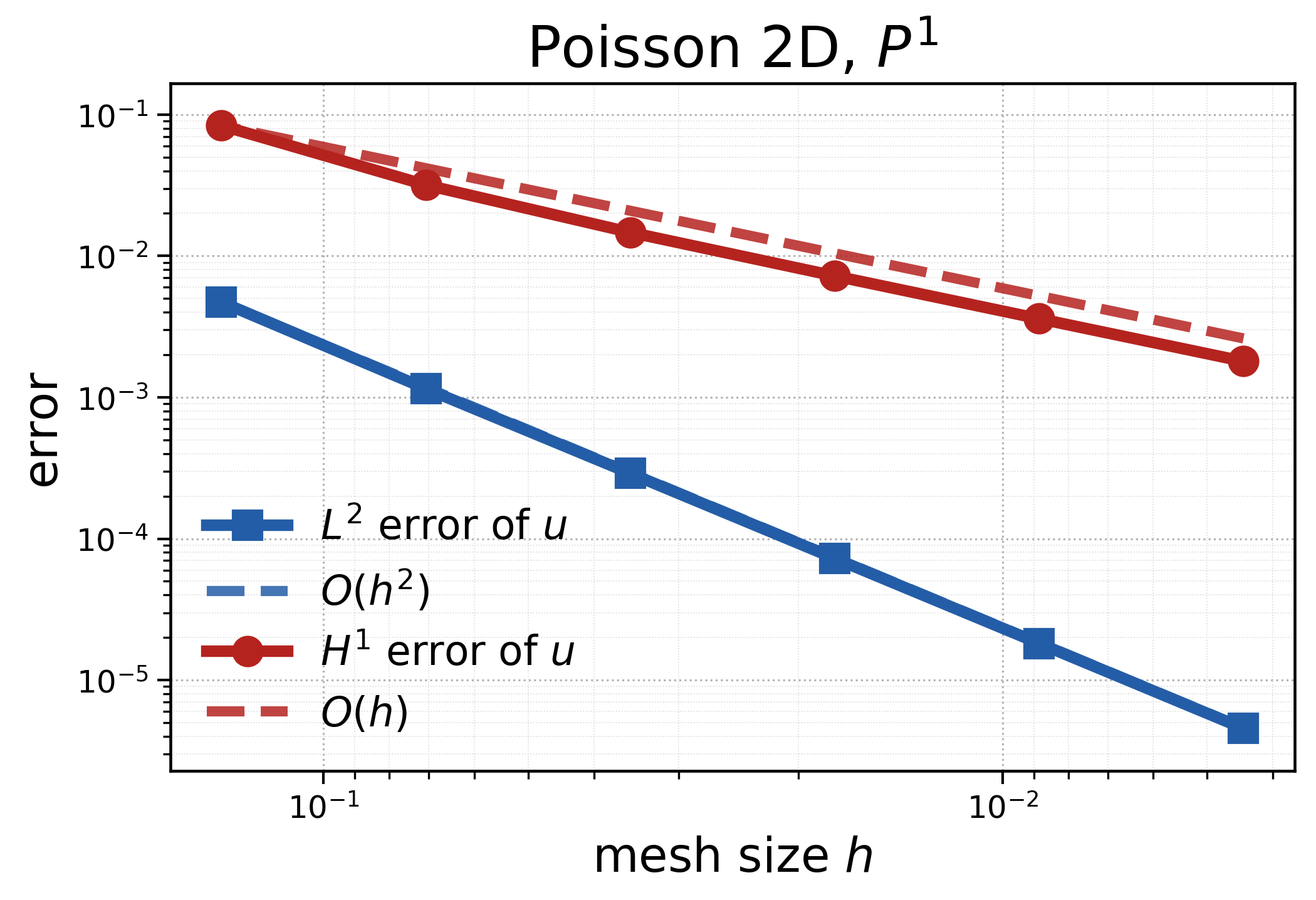}
        \caption{Poisson with $P^1$ elements}
        \label{fig1}
    \end{subfigure}\hfill
    \begin{subfigure}{0.48\textwidth}
        \centering
        \includegraphics[width=\linewidth]{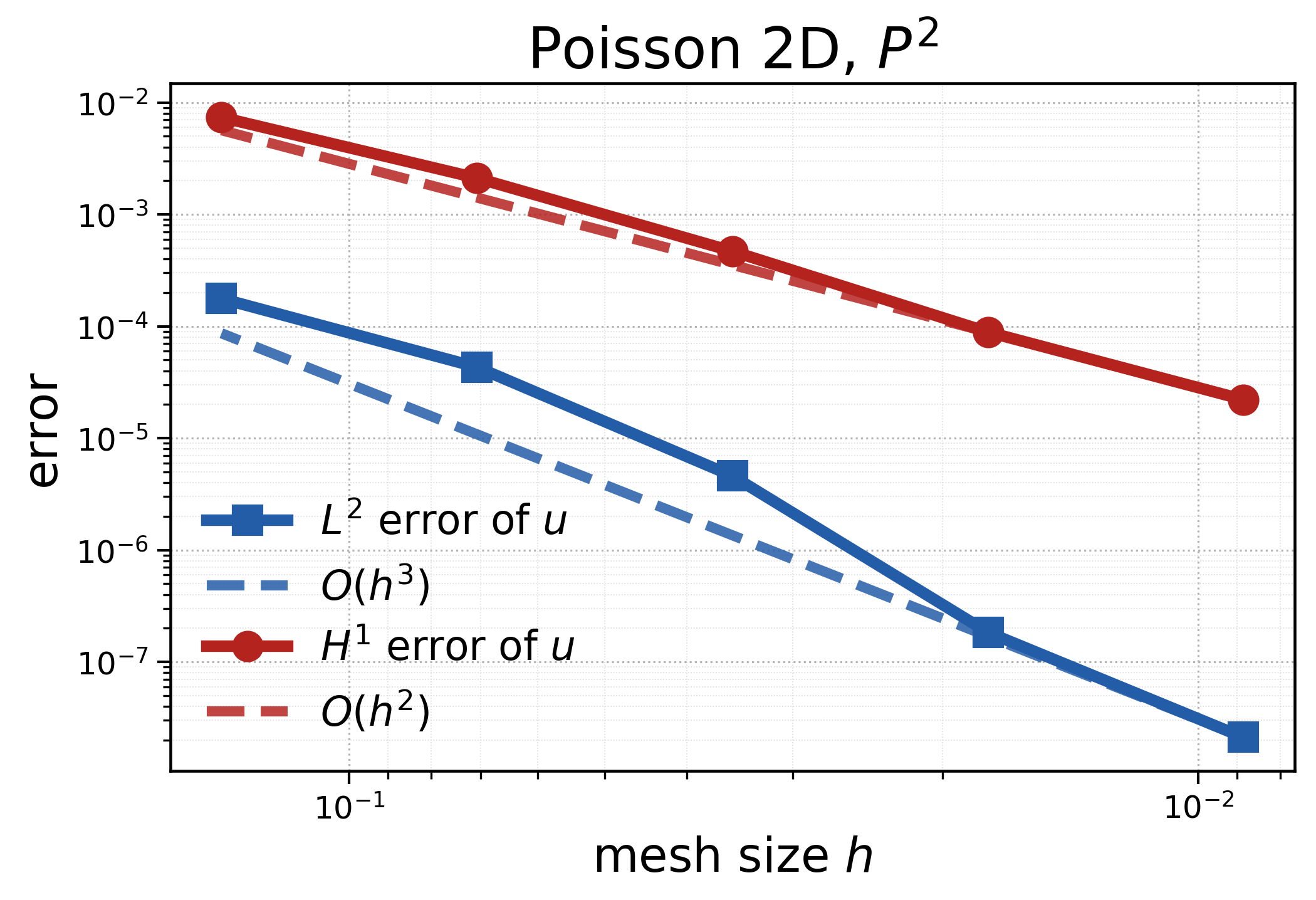}
        \caption{Poisson with $P^2$ elements}
        \label{fig2}
    \end{subfigure}
    \caption{Poisson with $P^1$ and $P^2$ elements}
\end{figure}
\FloatBarrier
\subsection{Poisson Equation in Three Dimensions}

In this example, we investigate the convergence behavior of the scheme \eqref{2.7} for the three-dimensional Poisson problem \eqref{Poisson}.
To this end, we consider a spherical computational domain $\Omega$ defined by
\[
\Omega = \{(x,y,z)\in\mathbb{R}^3 : (x-0.5)^2 + (y-0.5)^2 + (z-0.5)^2 < \frac18\}.
\]
and utilize a manufactured exact solution
\begin{align*}
u_{\mathrm{exact}}
&=\bigl(\tfrac18-(x-0.5)^2-(y-0.5)^2-(z-0.5)^2\bigr)
  \exp(x+2y-3z)\sin(2\pi y)\\
&\quad{}\times\cos(2\pi z+3\pi x).
\end{align*}
We then define $f_{\mathrm{exact}} = -\Delta u_{\mathrm{exact}}$ accordingly.

We present the convergence results for the implementation of the scheme \eqref{2.7} using $P^1$ elements.
As shown in Figure~\ref{fig3}, the numerical results demonstrate that the proposed method achieves optimal convergence rates in both the $L^2$ and $H^1$ norms.

\begin{figure}[!htbp]
    \centering
    \begin{subfigure}{0.48\textwidth}
        \centering
        \includegraphics[width=\linewidth]{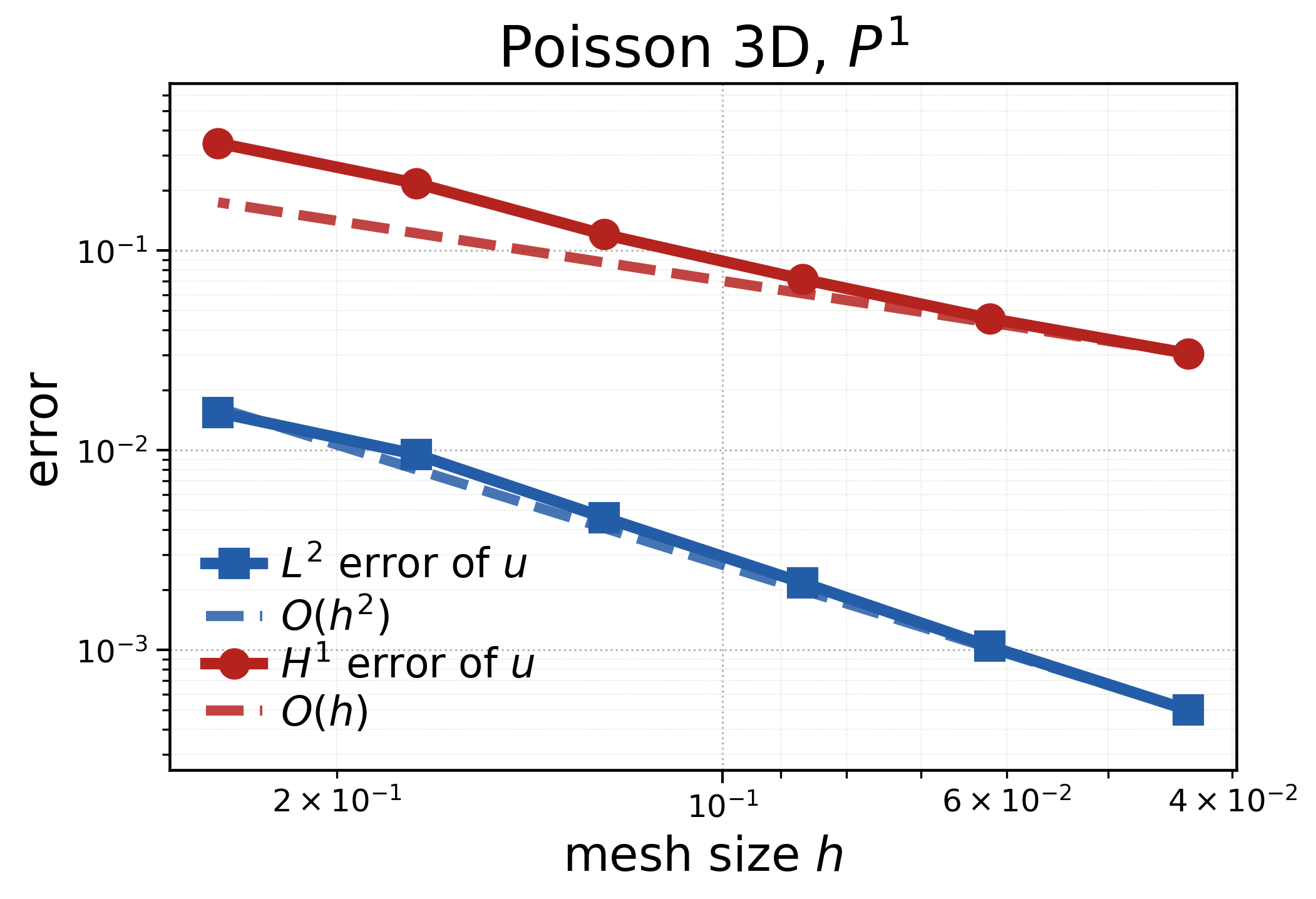}
        \caption{Poisson (3D) with $P^1$ elements}
        \label{fig3}
    \end{subfigure}\hfill
    \begin{subfigure}{0.48\textwidth}
        \centering
        \includegraphics[width=\linewidth]{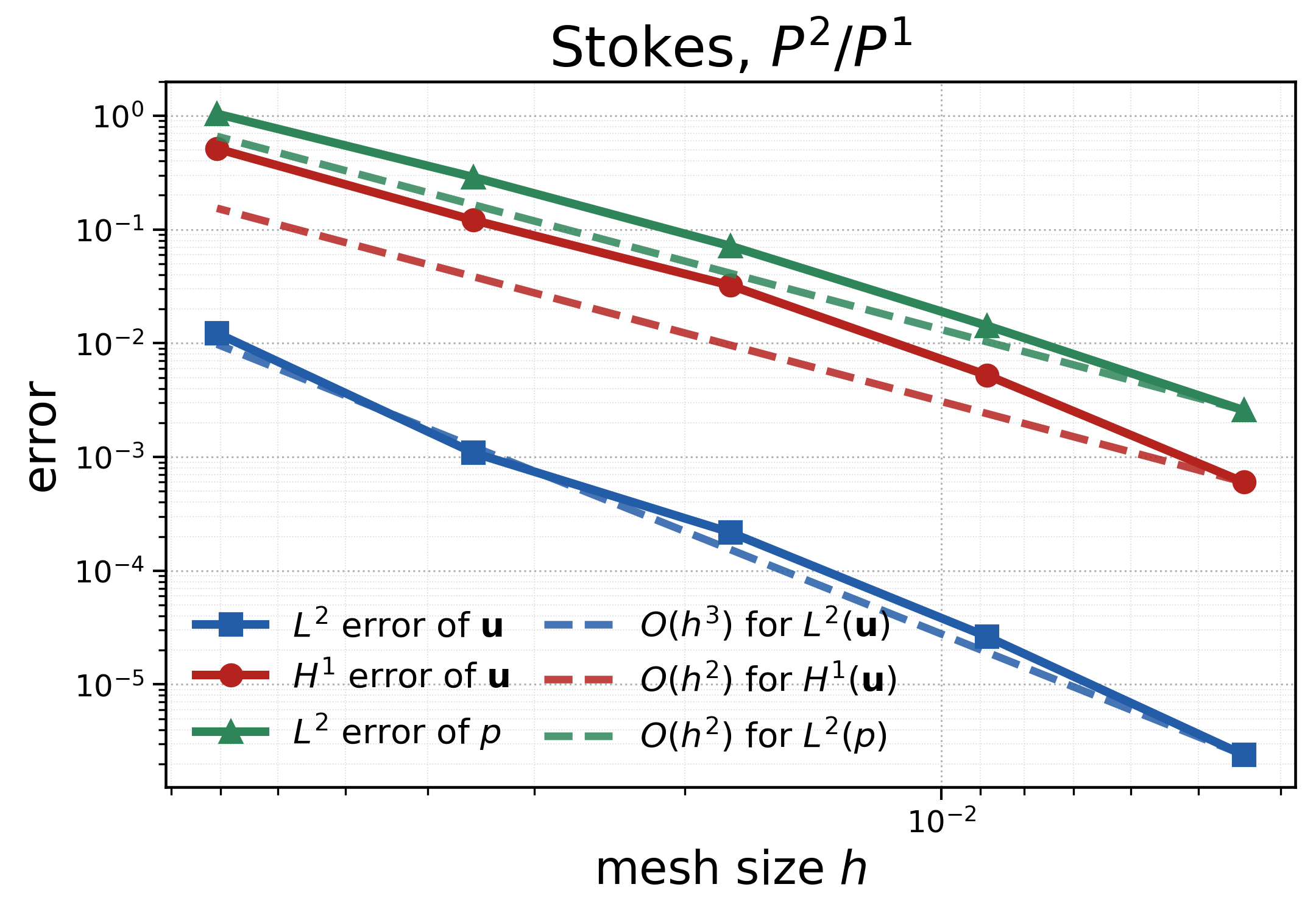}
        \caption{Stokes with $P^2/P^1$ elements}
        \label{fig4}
    \end{subfigure}
    \caption{Convergence results for the 3D Poisson and Stokes problems}
\end{figure}
\FloatBarrier

\subsection{Stokes Equations in Two Dimensions}
In this example, we consider the Stokes problem on the circular domain
\[
\Omega = \{(x,y)\in\mathbb{R}^2 : (x-0.5)^2 + (y-0.5)^2 < 0.25\}.
\]
A sequence of uniform triangular meshes with mesh size $h$ is generated on the background domain $[-1,2]\times[-1,2]$.
The exact velocity $u_{\mathrm{exact}}$ and pressure $p_{\mathrm{exact}}$ are well defined on $[-1,2]\times[-1,2]$ and are prescribed as
\begin{align}
u_{\mathrm{exact}}
&=
\Bigl(
 (y-0.5)\bigl((x-0.5)^2+(y-0.5)^2-0.25\bigr), \notag\\
&\qquad
-(x-0.5)\bigl((x-0.5)^2+(y-0.5)^2-0.25\bigr)
\Bigr),\\
p_{\mathrm{exact}}
&=
(x-0.5)+(y-0.5)+50(x-0.5)^3+\sin(100(x-0.5)).
\end{align}
The corresponding source term is given by
\[
f_{\mathrm{exact}} = -\Delta u_{\mathrm{exact}} + \nabla p_{\mathrm{exact}}.
\]

The numerical results are reported in Figure~\ref{fig4}.  
As can be observed, the numerical errors are in good agreement with the theoretical analysis. In particular, the velocity approximation achieves the optimal second-order convergence rate in the $H^1$ norm, while the pressure approximation converges with the optimal second-order rate in the $L^2$ norm, as explained in Remark~\ref{rem:projp}. These results confirm the stability and accuracy of the proposed scheme for the Stokes problem on unfitted meshes.

\subsection{Comparison with \texorpdfstring{$\phi$}{phi}-FEM} We perform a numerical experiment comparing the original
$\phi$-FEM with the proposed method in order to investigate their
robustness with respect to the stabilization parameter.

In the original \(\phi\)-FEM formulation for the Stokes problem in \cite{phifem_stokes}, two
stabilization parameters, \(\sigma\) and \(\sigma_u\), are used; in the
experiment, we take them to be the same parameter \(\sigma\). For the proposed method, we introduce a stabilization-weight
parameter $\sigma>0$ by multiplying both the stabilization block
and its matching consistency term in \eqref{3.7} by $\sigma$. More
precisely, the stabilization terms on the left-hand side become
\[
\sigma\left[
h^2\sum_{T\in\mathcal T_h^\Gamma}
\bigl(
-2\operatorname{div}Du_h+\nabla p_h,
-2\operatorname{div}Dv_h+\nabla q_h
\bigr)_{L^2(T)}
+j_h^{(2)}(u_h,v_h)
\right],
\]
while the corresponding term on the right-hand side becomes
\[
\sigma h^2\sum_{T\in\mathcal T_h^\Gamma}
\bigl(
f,-2\operatorname{div}Dv_h+\nabla q_h
\bigr)_{L^2(T)}.
\]
All other terms in \eqref{3.7} remain unchanged.

On the background domain $[-1, 2]\times[-1,2]$, we consider the Stokes equations with the following domain, and exact solution:
\[
\begin{aligned}
  \Omega
  &=
  \left\{(x,y)\in(0,1)^2:
  \left(x-\tfrac12\right)^2+\left(y-\tfrac12\right)^2<\tfrac14
  \right\},
  \qquad
  \phi
  =
  \left(x-\tfrac12\right)^2+\left(y-\tfrac12\right)^2-\tfrac14,\\
  u_{\mathrm{ex}}
  &=
  \left(
    \left(y-\tfrac12\right)\phi,
    -\left(x-\tfrac12\right)\phi
  \right),\\
  p_{\mathrm{ex}}
  &=
  \left(x-\tfrac12\right)+\left(y-\tfrac12\right)
  +50\left(x-\tfrac12\right)^3
  +\sin\!\left(100\left(x-\tfrac12\right)\right)
  +100\,
  \frac{\phi/0.01}{\sqrt{1+(\phi/0.01)^2}}.
\end{aligned}
\]
It follows that \(\nabla\cdot u_{\mathrm{ex}}=0\) in $\Omega$ and \(u_{\mathrm{ex}}=0\) on $\partial\Omega$. We then define
\(
  f=-\Delta u_{\mathrm{ex}}+\nabla p_{\mathrm{ex}}.
\)

For a fair comparison, both methods employ the same \(P^2/P^1\) Taylor--Hood finite element pair and the same active mesh \(\Omega_h\).
Computations are performed with \(h=\sqrt{2}/80\) and \(h=\sqrt{2}/240\). Since the
discrete pressure is subject to a zero-mean constraint over \(\Omega_h\),
we normalize the exact pressure accordingly by setting
\(p_{\mathrm{ex}}^0
  =
  p_{\mathrm{ex}}
  -
  \frac{1}{|\Omega_h|}
  \int_{\Omega_h} p_{\mathrm{ex}}\,dx.
\)
We then compute the relative velocity and pressure errors
\[
\begin{aligned}
  E_u
  &=
  \frac{
    \|u_h-u_{\mathrm{ex}}\|_{H^1(\Omega_h)}
  }{
    \|u_{\mathrm{ex}}\|_{H^1(\Omega_h)}
  },
  &
  E_p
  &=
  \frac{
    \|p_h-p_{\mathrm{ex}}^0\|_{L^2(\Omega_h)}
  }{
    \|p_{\mathrm{ex}}^0\|_{L^2(\Omega_h)}
  }.
\end{aligned}
\]

The results in the following figure demonstrate an advantage of the proposed method beyond the particular choice \(\sigma=1\). When \(\sigma\) is small,
the errors produced by the original \(\phi\)-FEM increase significantly,
whereas those of the proposed method remain controlled. For large values of
\(\sigma\), the accuracy of the original \(\phi\)-FEM for the pressure variable also deteriorates rapidly. This difference is particularly pronounced for the pressure approximation, for which the proposed method is more accurate over a broad
range of values of \(\sigma\). These results indicate that our method is more robust with respect to the choice of the stabilization parameter and therefore requires less parameter tuning.

\begin{figure}[!htbp]
  \centering
  \includegraphics[width=0.95\linewidth]
  {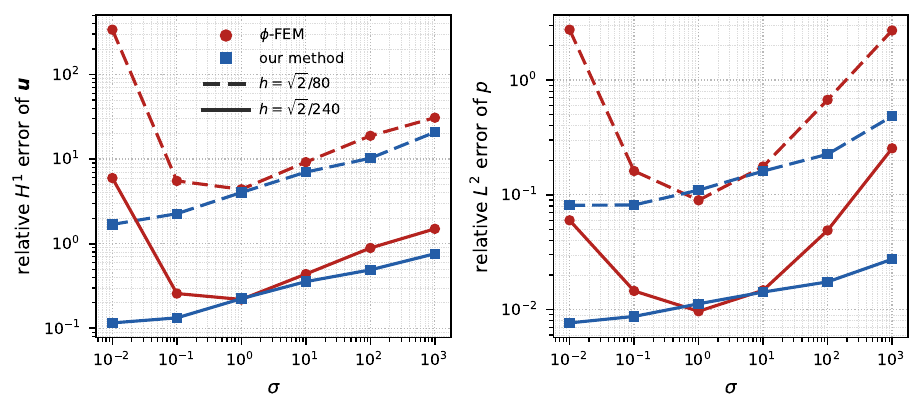}
  \caption{Sensitivity of the relative velocity \(H^1\)-error
  (left) and pressure \(L^2\)-error (right) to the stabilization
  parameter \(\sigma\), for the original \(\phi\)-FEM and the
  proposed method with \(h=\sqrt{2}/80\) and
  \(h=\sqrt{2}/240\).}
  \label{fig:parameter-sensitivity-all-errors}
\end{figure}
\FloatBarrier
\section{Conclusion}

We have introduced and analyzed a novel unfitted finite element method, called the $H^{-1}$ least-squares UnCut FEM, for the Poisson and Stokes equations on domains defined by a level set function. The proposed method inherits the implementation advantages of the existing $\phi$-FEMs, while the design of the least-squares formulation eliminates the need for the stabilization parameters to be sufficiently large, thereby improving both robustness of computation and simplicity of implementation. The formulation in the $H^{-1}$ setting is the key to enabling the use of $C^0$ finite elements while retaining these favorable properties. Optimal-order $H^1$-convergence of the velocity and $L^2$-convergence of the pressure (up to a constant) are rigorously established for the Stokes equations, together with optimal-order $H^1$-convergence for the Poisson equation.

Several directions remain for future research, including the theoretical derivation of optimal-order $L^2$-error estimates as well as extensions to Neumann/Robin boundary conditions and time-dependent problems. 


\bibliographystyle{siam}

\bibliography{LeastSquare_Phi}
\end{document}